\documentclass{amsart}
\usepackage{amsmath}
\usepackage{amsfonts}
\usepackage{amssymb}
\usepackage{graphicx}
\usepackage{color}
\usepackage{amsmath}
\usepackage{amsfonts}
\usepackage{amssymb}
\usepackage{graphicx}
\usepackage{xypic}
\usepackage{tikz-cd}
\usepackage{color}
\usepackage{mathbbol}
\providecommand{\U}[1]{\protect\rule{.1in}{.1in}}
\newtheorem{theorem}{Theorem}[section]

\newtheorem{corollary}[theorem]{Corollary}

\newtheorem{definition}[theorem]{Definition}

\newtheorem{lemma}[theorem]{Lemma}

\newtheorem{proposition}[theorem]{Proposition}
\newtheorem{remark}[theorem]{Remark}

\newtheorem{construction}[theorem]{Construction}
\numberwithin{equation}{section}

\title{On the (Local) Lifting Property}
\author{Dominic Enders, Tatiana Shulman}

\begin{document}

\maketitle

\begin{abstract}  The (Local) Lifting Property  ((L)LP) is introduced by Kirchberg and deals with lifting completely positive maps. We give a characterization of the (L)LP in terms of lifting $\ast$-homomorphisms. We use it to  prove that if separable $A$ and $B$ have the LP and $F$ is their finite-dimensional C*-subalgebra, then $A\ast_F B$ has the LP. This answers a question of Ozawa \cite{Ozawa}.

We prove that Exel's soft tori have the LP. As a consequence we obtain that $C^*(F_n\times F_n)$ is inductive limit of RFD C*-algebras with the LP.

We prove that for a class of C*-algebras including  $C^*(F_n\times F_n)$, all contractible C*-algebras and all suspensions, the LLP is equivalent to Ext being a group.

As byproduct of methods developed in the paper we generalize Kirchberg's theorem about extensions with the WEP,  give short proofs of several, old and new,  facts about soft tori, new unified proofs of Li and Shen's characterization of RFD property of free products amalgamated over a finite-dimensional subalgebra and Blackadar's characterization of semiprojectivity of them.
%   and we show that the class of C*-algebras with the LP is closed under cross products with amenable groups.

\end{abstract}

\section{Introduction}

The celebrated Choi-Effros Theorem  \cite{ChoiEffros} states that if $A$ is a nuclear separable C*-algebra, then  any contractive completely positive map from $A$ to a quotient C*-algebra admits a contractive completely positive lifting:
$$ \begin{tikzcd}  & B \arrow[d]  \\ A   \arrow[r] \arrow[ru, dashed] & B/I  \end{tikzcd}$$

In his outstanding paper \cite{Kirchberg}  Kirchberg introduced the (Local) Lifting Property ((L)LP) that requires, for a C*-algebra $A$,  any c.c.p. map  from $A$ to a quotient C*-algebra to  lift (locally) to a c.c.p. map
(see the section Preliminaries for precise definitions). The (L)LP has deep connections with tensor products and the weak expectation property of C*-algebras (\cite{Kirchberg}, \cite{PisierTensor})  and is one of key ingredients in the proving the equivalence of several fundamental conjectures, including Connes Embedding Problem and Tsirelson's problem (\cite{Kirchberg}, \cite{PisierBook} and references therein).

Despite of the fact that the properties LP and LLP are of central importance, there are not many examples of C*-algebras with the (L)LP  outside the class of nuclear C*-algebras.
In \cite{Kirchberg} Kirchberg proved that the (L)LP is preserved under certain operations, e.g.  both  LP and LLP  are closed under tensoring with nuclear C*-algebras, and the LLP is closed under extensions. Boca proved that the LP is closed under free products \cite{Boca} and Pisier proved the same for the LLP \cite{PisierFreeProduct}. The latter result was generalized by Ozawa who showed that the LLP is closed   under an
amalgamated free product over a finite dimensional C*-subalgebra \cite{Ozawa}.  Ozawa states that it is not known whether the LP is preserved under an amalgamated free product over a finite dimensional C*-subalgebra (\cite[p. 15]{Ozawa}).

\medskip

In this paper we further study the (L)LP.
In section 3 we give a characterization of the (L)LP in terms of lifting $\ast$-homomorphisms  (here by lifting of a $\ast$-homomorphism we mean lifting it to a $\ast$-homomorphism). In general  $\ast$-homomor-phisms lift more rare than ccp maps. A C*-algebra $A$ such that any $\ast$-homomorphism from $A$ to a quotient  C*-algebra lifts is called {\it projective}, and  projectivity is a very strong and therefore rare property \cite{LoringBook}. That for liftability of ccp maps it is sufficient to lift $\ast$-homomorphisms was observed already in the proof that $C^*(F_n)$ has the LP (\cite[Th. 13.1.3]{BO}). This observation was stated explicitly in \cite[Prop. 6.6]{CourtneySherman}  and \cite[Cor. 2.9]{Courtney}.
Here we prove that to be able to lift ccp maps,  a certain lifting condition for $\ast$-homomorphisms is  not only sufficient but also necessary.

\medskip

\noindent {\bf  Theorem} {\it  Let $A$ be separable. The following are equivalent:

\medskip
(i) $A$ has the LP;
\medskip

(ii) for any $\sigma$-unital C*-algebra $B$, its ideal $I$ and a $\ast$-homomorphism $f: A \to M(B/I\otimes K)$ there is a $\ast$-homomorphism $g: A \to M(B/I \otimes K)$ such that $f\oplus g$ lifts to a $\ast$-homomorphism  $A\to M(B\otimes K)$;

  \medskip

(iii) for any $\sigma$-unital C*-algebra $B$, its ideal $I$ and a $\ast$-homomorphism $f: A \to M(B/I\otimes K)$ there is a $\ast$-homomorphism $g: A \to M(B/I \otimes K)$ such that $f\oplus g$ lifts to a ccp map  $A\to M(B\otimes K)$.
}

\medskip

 We also give a reformulation of the LLP in terms of lifting $\ast$-homomorphisms.
 
 \medskip

\noindent {\bf  Theorem} {\it Let $A$ be separable. The following are equivalent:

\medskip
(i) $A$ has the LLP;
\medskip

(ii) for any separable C*-subalgebra $D$ of $Q(H)$ and any $\ast$-homomorphism $f: A \to M(D\otimes K)$ there is a $\ast$-homomorphism $g: A \to M(D \otimes K)$ such that $f\oplus g$ lifts to a $\ast$-homomorphism  $A\to M(\pi^{-1}(D)\otimes K)$;

  \medskip

(iii) for any separable C*-subalgebra $D$ of $Q(H)$ and any $\ast$-homomorphism $f: A \to M(D\otimes K)$ there is a $\ast$-homomorphism $g: A \to M(D \otimes K)$ such that $f\oplus g$ lifts to a ccp map  $A\to M(\pi^{-1}(D)\otimes K)$.}

\medskip

A very useful result on c.p. liftings is Arveson's theorem that says that the set of liftable ccp maps is closed in the pointwise operator norm topology \cite{Arveson}. As an auxiliary result needed to prove the characterization of the LLP above  we give  a generalization, to the strict topology,  of Arveson's theorem which might be of independent interest ({\bf Theorem \ref{strict} }).   

\medskip

 The result of Boca and in separable case  also the result of Pisier on free products follow from the above characterizations of the LP and LLP immediately. Our characterization of the LP is also the main tool that allows us to settle the aforementioned question about free products amalgamated
over a finite-dimensional C*-subalgebra.

\medskip

\noindent {\bf Theorem} {\it Let $F$ be a finite-dimensional C*-subalgebra of separable C*-algebras $A$ and $B$. If $A$ and $B$ have the LP, then  $A\ast_F B$ has the LP.}

\medskip

This theorem is proved in section 4.  As a consequence we obtain that finite tree products with finite edge groups and countable vertex groups, in particular finitely generated virtually free groups, have full group C*-algebras with the LP ({\bf Corollary \ref{VirtuallyFree}}).  Our technique for amalgamated free products applies also to some other lifting properties. We give a new proof of Blackadar's result stating that semiprojectivity passes to  free products amalgamated over a finite-dimensional subalgebra ({\bf Corollary \ref{SP}}). We give a new proof of Li and Shen's characterization of when unital free products amalgamated over a finite-dimensional subalgebra are RFD ({\bf Theorem \ref{RFDunital}}). We also obtain a characterization in the non-unital case ({\bf Theorem \ref{RFDnonunital}}).

\medskip

In section 5 we prove that Exel's soft tori have the LP ({\bf Theorem \ref{SoftTorusthe LP}}). The soft torus $C(\mathbb T^2)_{\epsilon}$ is the universal C*-algebra generated by two unitaries commmuting up to $\epsilon$:
$$C(\mathbb T^2)_{\epsilon} = C^*\langle u, v\;|\; u \; \text{and} \; v \; \text{are unitaries and} \; \|[u, v]\| \le \epsilon\rangle.$$
Soft tori play an important role in the study of almost commuting matrices and operators,  their properties were extensively studied.  Here we develop a simple method which allows first of all to prove the LP and at the same time to give a unified proof of some of the known results, and obtain some new properties. Namely in section 6 we give a short proof of Exel's result \cite{Exel} that the soft tori form a continuous field   ({\bf Theorem \ref{ContinuousField}}),  Eilers and Exel' result  \cite{EilersExel} that soft tori are RFD   ({\bf Theorem \ref{SoftTorusRFD}}), we show that soft tori are very flexibly stable in the sense of Becker and Lyubotzky  ( {\bf Proposition \ref{SoftTorusVeryFlexiblyStable}}) and we strengthten  Exel's result \cite{Exel} on *-strong dilated perturbations of unitary pairs ({\bf Theorem \ref{dilations}}).

The full group C*-algebra  $C^*(F_2\times F_2)$  plays an important role in the C*-algebra theory. The Connes Embeding Problem is equivalent to the question of whether or not $C^*(F_2\times F_2)$ is RFD. Whether  $C^*(F_2\times F_2)$ has the (L)LP is an outstandng open question. Using the result of section 5 we show that $C^*(F_2\times F_2)$ is inductive limit of RFD C*-algebras with the LP ({\bf Proposition \ref{OtherPropertiesOfFnxFn}}).

\medskip

Section 7  is devoted to the semigroup $Ext$ of extensions by compact operators. The semigroup $Ext$ was introduced in the celebrated work of Brown-Douglas-Fillmore \cite{BDF}. The first example of a C*-algebra $A$ such that $Ext(A)$ is not a group was constructed  by  J. Anderson \cite{Anderson}. Since then more examples were found but still  there is no clear understanding of when $Ext$ is a  group.  It is known that the LLP implies that $Ext$ is a group and it is not known whether these properties are equivalent (\cite[p.12]{Ozawa}). In this section we prove that for an interesting class of C*-algebras, including $C^*(F_n\times F_n)$,  all separable  contractible C*-algebras and all separable suspensions, they are equivalent.

\medskip

\noindent {\bf Theorem} {\it Suppose $A$ is an inductive limit,  with surjective connecting maps,  of separable C*-algebras that are RFD and have the LP (e.g. $A = C^*(F_n\times F_n)$ or $A$ is any contractible C*-algebra or $A$ is any suspension). Then $A$ has the LLP if and only if $Ext(A)$ is a group.}

\medskip

Kirchberg  proved that a unital separable C*-algebra $A$ has the LLP if
and only if $Ext(SA)$ is a group if and only if $Ext(cone(A))$ is a group \cite{Kirchberg}. Since   $cone(A)$ (and $SA$ ) has the LLP if and only if $A$ has,  Kirchberg's theorem follows immediately from the  theorem above. Moreover  it implies that in Kirchberg's theorem the unitality assumption is not needed.

As mentioned above it is not known whether $C^*(F_2\times F_2)$ has the LLP, but now we can say that the problem is equivalent to the question of whether $Ext\left( C^*(F_2\times F_2)\right)$ is a group.

\medskip

The last section of the paper contains miscellaneous results:

\medskip

1)  First, we give a generalization of Kirchberg's theorem on cones \cite[Th. 13.4.1]{BO} that says that  if $A$ is a separable QWEP C*-algebra, then  $cone(A)$ has a quasidiagonal extension with the WEP. Recall that a C*-algebra is said to have the {\it weak expectation property (WEP)} if there exists a u.c.p. map $\Phi: B(H_u) \to A^{**}$ such that $\Phi(a) = a$ for all $a\in A$, where $A\subset A^{**} \subset B(H_u)$ is the universal representation of $A$. A C*-algebra is called {\it QWEP} if it is a quotient of a C*-algebra with the WEP.  The Connes Embdedding Problem is equivalent to the statement that each separable C*-algebra is QWEP (\cite{Kirchberg}).

\medskip

\noindent {\bf Theorem}   {\it  Let $A$ be  separable contractible (or, more generally, an inductive limit,  with surjective connecting maps,  of separable RFD  C*-algebras that have  the LP) with the QWEP. Then there exists a quasidiagonal
    extension
    $$0\to K \to B \to A \to 0$$  such that $B$ has the WEP.}

\medskip

\noindent Since $A$ is QWEP if and only if $cone(A)$ is QWEP,  Kirchberg's theorem on cones is a particular case of the theorem above.

\medskip

2) In {\bf Theorem 8.3} we give a new and short  proof to the fact that the class of C*-algebras with the LP is closed under crossed products with amenable groups which was proved by Buss, Echterhoff and Willett \cite[Th. 7.4]{BEW}.
 %These are free groups and everything than can be obtained from them and amenable groups by direct products with amenable groups,  by taking subgroups, by free products and, as we now know from the result of section 4, by free products amalgamated over finite subgroups.

 %During the workshop “Amenability, coarse embeddability and fixed point properties” at MSRI in December 2016, Pisier explicitly stated the problem of finding more full group C*-algebras either with or without the (L)LP. Since then many full group C*-algebras without the (L)LP we were found (\cite{Ioana}).

%\medskip

%\noindent {\bf Theorem}()   {\it Suppose $G$ is an amenable group  and $A$ is a separable C*-algebra with the LP. Then $A\rtimes G$ has the LP. In particular, $C^*(F_n\rtimes G)$ has the LP, for any amenable group $G$.}

\medskip

\noindent {\bf Acknowledgments. }
 We thank Alcides Buss for informing us about \cite[Th. 7.4]{BEW} and for comments on the first version of this work. T.S. is partially supported by a grant from the Swedish Research Council.

\section{Preliminaries}

\begin{definition} Let $A$ be a C*-algebra, $I$ be an ideal\footnote{By an ideal we always will mean a closed two-sided ideal.} in a C*-algebra $B$, $q: B \to B/I$ be the quotient map.  We say that a ccp map $\phi: A\to B/I$ is liftable  if there exists a ccp map $\psi: A\to B$ such that $q\circ \psi = \phi$.  We say that $\phi$ is locally liftable if for any finite-dimensional operator system $E\subset A$, there exists a ccp map $\psi: E \to B$ such that
$q\circ \psi = \phi|_E$.
\end{definition}

\begin{definition}\label{defLP} A unital C*-algebra $A$ has the lifting property (the LP)  (respectively the local lifting property (the LLP)) if any ccp map from $A$ into a quotient C*-algebra $B/I$ is liftable (respectively locally liftable). A non-unital C*-algebra has the LP (the LLP, respectively)  if its unitization has that property.
\end{definition}

By \cite[Lemma 13.1.2]{BO} to prove that a unital C*-algebra  has the LP (respectively the LLP) it is sufficient to lift only ucp maps, as opposed to ccp.

\medskip

We will often use the Noncommutative Tietze Extension Theorem (\cite{Pedersen}) which says that  any surjective $\ast$-homomorphism $\pi: B\to A$ extends to a $\ast$-homomorphism $\tilde\pi: M(B)\to M(A)$ that is continuous with respect to the strict topology on $M(B)$ and $M(A)$, and if $B$ is $\sigma$-unital, then $\tilde \pi$ is surjective.

Let $K$ denote the C*-algebra of all compact operators on a separable Hilbert space (sometimes we also will denote it by $K(H)$). Given a surjection $B\to B/I$ with separable $B$ we will  apply the Noncommutative Tietze Extension Theorem to extend the surjection $B\otimes K \to B/I \otimes K$  to a  surjection $M(B\otimes K)  \to M(B/I \otimes K)$.

%We also will use Kasparov-Steinspring Dilation Theorem \cite{KasparovDilation} which says that if $A$ is separable, $B$ is $\sigma$-unital, and $\phi: A \to B$ is ccp, then $\phi$ dilates to a $\ast$-homomorphism from $A$ to $M(B\otimes K$.

The Calkin algebra $B(H)/K(H)$ will be denoted by $Q(H)$.

\section{A characterization of the (L)LP in terms of lifting homomorphisms}

\subsection{A characterization of the LP}

\begin{theorem} \label{CharacterizationLP} Let $A$ be separable. The following are equivalent:

\medskip
(i) $A$ has the LP;
\medskip

(ii) for any $\sigma$-unital C*-algebra $B$, its ideal $I$ and a $\ast$-homomorphism $f: A \to M(B/I\otimes K)$ there is a $\ast$-homomorphism $g: A \to M(B/I \otimes K)$ such that $f\oplus g$ lifts to a $\ast$-homomorphism  $A\to M(B\otimes K)$;

  \medskip

(iii) for any $\sigma$-unital C*-algebra $B$, its ideal $I$ and a $\ast$-homomorphism $f: A \to M(B/I\otimes K)$ there is a $\ast$-homomorphism $g: A \to M(B/I \otimes K)$ such that $f\oplus g$ lifts to a ccp map  $A\to M(B\otimes K)$.

%(iv) for any $\sigma$-unital C*-algebra $B$, its ideal $I$ and a $\ast$-homomorphism $f: A \to B/I$ there is a $\ast$-homomorphism $g: A \to M(B/I \otimes K)$ such that $f\oplus g$ lifts to a ccp map  $A\to M(B\otimes K)$.
\end{theorem}

\begin{proof}
 (i) $\Rightarrow$ (ii): let $f: A \to M(B/I\otimes K)$ be a $\ast$-homomorphism. Since $A$ has the LP, $f$ lifts  to a ccp map $\Phi: A\to M(B\otimes K)$. By Kasparov-Stinespring Dilation Theorem \cite{KasparovDilation}, $\Phi$ dilates to a $\ast$-homomorphism $F: A \to M(B\otimes K)$.
  In other words, $F$ is of the form
$$F = \left(\begin{array}{cc} \Phi & \Phi_2\\ \Phi_3 & \Phi_4\end{array}\right).$$ Then
$$\left(\begin{array}{cc} \Phi(ab) & \Phi_2(ab)\\ \Phi_3(ab) & \Phi_4(ab)\end{array}\right) = F(ab) = F(a)F(b) = \left(\begin{array}{cc} \Phi(a) & \Phi_2(a)\\ \Phi_3(a) & \Phi_4(a)\end{array}\right)\left(\begin{array}{cc} \Phi(b) & \Phi_2(b)\\ \Phi_3(b) & \Phi_4(b)\end{array}\right).$$
This in particular implies that
\begin{equation}\label{1} \Phi(ab) = \Phi(a)\Phi(b) + \Phi_2(a)\Phi_3(b) \end{equation} and
\begin{equation}\label{2} \Phi_4(ab) = \Phi_3(a)\Phi_2(b) + \Phi_4(a)\Phi_4(b),\end{equation} for any $a, b\in A$.
Since $f = \pi(\Phi)$ is a $\ast$-homomorphism, (\ref{1}) implies that
\begin{equation}\label{4} \pi(\Phi_2(a))\pi(\Phi_3(b)) = 0, \end{equation} for any $a, b\in A$. Since $F(a) = F(a^*)^*$,  we have
\begin{equation}\label{3} \Phi_2(a) = (\Phi_3(a^*))^*.\end{equation} By taking $b = a^*$ and substituting (\ref{3}) into (\ref{4}), we obtain that $\pi\circ\Phi_2=\pi\circ\Phi_3= 0$. Together with (\ref{2}) this implies that $g:= \pi\circ \Phi_4$ is a $\ast$-homomorphism, and $f\oplus g$ lifts to the $\ast$-homomorphism $F$.

\medskip

(ii) $\Rightarrow$ (iii): obvious.

\medskip

(iii) $\Rightarrow$ (i): let $\phi: A \to B/I$ be a ccp map. We can assume that $B$ is separable. Indeed, let $a_1, a_2, \ldots$ be a dense subset of $A$ and  let $b_i$ be a preimage of $\phi(a_i)$, $i\in \mathbb N$. Let $B_0$ be the C*-subalgebra of $B$  generated by $b_1, b_2, \ldots$, and $I_0 = I\cap B_0$. Then $\phi$ can be considered as a map to $B_0/I_0$.

By Kasparov-Stinespring Dilation Theorem, $\phi$ dilates to a $\ast$-homomorphism $f: A \to M(B/I \otimes K)$. There is a $\ast$-homomorphism $g: A \to M(B/I \otimes K)$ such that $f\oplus g$ lifts to a ccp map $F: A \to M(B\otimes K)$.  If $B$ is unital, then the corner $(F_{11})_{11}: A \to B$ is a ccp lift of $\phi$. If $B$ is not unital, then the corner $h: = (F_{11})_{11}: A \to M(B)$ is ccp
and $\tilde q\circ h = \phi$.  Here  $\tilde q: M(B) \to M(B/I)$ is the quotient map obtained from the quotient map $q: B\to B/I$ by the NC Tietze Extension Theorem.  Let $(e_n)$ be an approximate unit for $B$. Let 
$$h_n(a) = e_n^{1/2}h(a)e_n^{1/2} \in B,$$ $a\in A$. Then 
$$q(h_n(a)) = q(e_n)^{1/2}\phi(a)q(e_n)^{1/2}\to \phi(a) \;\text{in norm},$$ $a\in A$. Hence by Arveson's theorem $\phi  $ has a ccp lift.
\end{proof}

As follows from what is mentioned after Definition \ref{defLP}, when $A$ is unital, all $\ast$-homomorphisms in  Theorem \ref{CharacterizationLP} can be replaced by unital $\ast$-homomorphisms. 

\medskip

\subsection{A characterization of the LLP}

Let $\pi: B(H) \to Q(H)$ be the canonical surjection. In what follows we will often use the fact that a separable C*-algebra $A$ has the LLP if and only if every ccp map $A\to Q(H)$ lifts to a ccp map $A\to B(H)$ (\cite[Prop. 3.13]{Ozawa}).

\begin{lemma}\label{DtensorK} Let $A$ be  a separable C*-algebra with the LLP. Then for any separable C*-subalgebra $D$ of $Q(H)$, any ccp map $\phi: A \to D\otimes K$ lifts to a ccp map from $A$ to $\pi^{-1}(D)\otimes K$.
\end{lemma}
\begin{proof} An isomorphism of $H$ and $H^{(n)}$ induces  isomorphisms $\theta_n: B(H)\to B(H)\otimes M_n$ and $\tilde \theta_n: Q(H)\to Q(H)\otimes M_n$ such that
\begin{equation}\label{tensorMatricesAgree}(\pi\otimes id_{M_n})\circ \theta_n = \tilde\theta_n\circ\pi.\end{equation}
Let $D$ be a  separable C*-subalgebra of $Q(H)$  and $\phi: A \to D\otimes K$ a ccp map. Let $P_n\in K$, $n\in \mathbb N$, be an increasing sequence of projections that strongly converges to $\mathbb 1$. We define $\phi_n: A \to D\otimes M_n\subset D\otimes K$ by
$$\phi_n(a) = (1_{Q(H)}\otimes P_n) \phi(a) (1_{Q(H)}\otimes P_n),$$
$a\in A$. Since $A$ has the LLP, $\tilde\theta_n^{-1}\circ \phi_n: A \to Q(H)$ lifts to a ccp map $\psi_n: A \to B(H)$. Then by (\ref{tensorMatricesAgree}), $\theta_n\circ\psi_n: A \to B(H)\otimes M_n \subset B(H) \otimes K$ is a ccp lift of $\phi_n$.

Since for any $T\in K$, $P_nTP_n\to T$,
$$\phi(a)= \lim_{n\to \infty} \phi_n(a), $$ for each $a\in A$. Then by Arveson's theorem, $\phi$ lifts to a some ccp map $\bar \phi: A\to B(H)\otimes K$. Since $(\pi\otimes id_K)^{-1}(D\otimes K) = \pi^{-1}(D)\otimes K,$ $\bar\phi$ is  a ccp map to $\pi^{-1}(D)\otimes K$.
\end{proof}

\medskip

%\subsection{Strict limits of liftable cp maps}

The following  theorem is a generalization, to the strict topology,  of Arveson's result that the set of liftable cp maps is closed in the pointwise operator norm topology. The arguments are similar to Arveson's.

\begin{theorem}\label{strict} Let $A$ and $B$ be separable C*-algebras and $I$ an ideal in $B$. Suppose $\phi_n: A \to M(B/I)$, $n\in \mathbb N$, and  $\phi: A \to M(B/I)$ are ccp maps and $\phi_n(a)\to \phi(a)$ in the strict topology, for any $a\in A$. Suppose each $\phi_n$ lifts to a ccp map. Then $\phi$ lifts to a ccp map.
\end{theorem}
\begin{proof} Let $\psi_n$ be a ccp lift of $\phi_n$, $n\in \mathbb N$. Let $a_1, a_2, \ldots$ be a dense subset of the unit ball of $A$ and $x_1, x_2, \ldots$ be a dense subset of the unit ball of $B$. We choose  a subsequence of $\{\phi_n\}$, which we also denote by $\{\phi_n\}$, such that
$$ \|\phi_{n+1}(a_i)\pi(x_k) - \phi_n(a_i)\pi(x_k)\| < \frac{1}{2^n},$$ when $i\le n, k\le n$.  Let $\{i_{\lambda}\}$ be a quasicientral approximate unit for $I \triangleleft B$. We will construct a ccp lift $\tilde \psi_n$ of $\phi_n$, $n\in \mathbb N$, such that
$$ \|\tilde\psi_{n+1}(a_i)x_k - \tilde\psi_n(a_i)x_k\| < \frac{1}{2^{n-2}},$$ when $i\le n, k\le n$.

Let $\tilde \psi_1 = \psi_1$. Assume $\tilde \psi_2, \ldots, \tilde \psi_n$ are already constructed. Let
$$\tilde\psi_{n+1} = (1-i_{\lambda})^{1/2} \psi_{n+1} (1-i_{\lambda})^{1/2} +  i_{\lambda}^{1/2}\tilde \psi_n  i_{\lambda}^{1/2}, $$ where $\lambda$ is sufficiently large so that
\begin{multline*} \|[ (1-i_{\lambda})^{1/2}, x_k]\|  \le \frac{1}{2^{n+1}},\; \|[ (1-i_{\lambda})^{1/2}, \psi_{n+1}(a_i)x_k]\|\le \frac{1}{2^{n+1}},\\ \| [i_{\lambda}^{1/2}, x_k] \| \le \frac{1}{2^{n+1}}, \;
\|[i_{\lambda}^{1/2}, \tilde\psi_n(a_i)x_k]\| \le \frac{1}{2^{n+1}}, \\ \|(1-i_{\lambda})\left(\psi_{n+1}(a_i)x_k - \tilde \psi_n(a_i)x_k\right)\| \le \frac{1}{2^{n-1}},\end{multline*}
when $i\le n, k\le n$. It is straightforward to check that $$ \|\tilde\psi_{n+1}(a_i)x_k - \tilde\psi_n(a_i)x_k\| < \frac{1}{2^{n-2}},$$ when $i\le n, k\le n$. Therefore for any $i, k$ the sequence $\tilde \psi_n(a_i)x_k$ converges. It implies that for any $a\in A$, $x\in B$ the sequence $\tilde\psi_n(a)x$ converges. By applying this to $a^*$ and taking adjoints, we obtain that the sequence $ x\tilde\psi_n(a)$  converges.
 By \cite[p.39]{WeggeOlsen} this implies that for each $a\in A$ the sequence $\tilde \psi_n(a)$ converges in the strict topology.  Therefore $\tilde\psi_n$ converge pointwisely in the strict topology to some map $\psi$. This $\psi$ must be ccp. Indeed if we embed $B$ non-degeneratly into $B(H)$, then the strict topology on $M(B)\subset B(H)$  is finer than the strong operator topology (\cite[Prop.2.3.4]{WeggeOlsen}) and therefore for any $(a_{ij})\in M_N(A)$ one has $\tilde\psi_n^{(N)}(a_{ij}) \to \psi^{(N)}(a_{ij})$ in the strong operator topology. Since the strong limit of positive contractive operators is positive and contractive, we conclude that $\psi$ is ccp.

Since $\tilde \pi$ is strictly continuous, $\psi$ is a lift of $\phi$.
\end{proof}

\medskip

Now we obtain a characterization of the LLP in terms of lifting $\ast$-homomorphisms.

\begin{theorem} \label{CharacterizationLLP} Let $A$ be separable. The following are equivalent:

\medskip
(i) $A$ has the LLP;
\medskip

(ii) for any separable C*-subalgebra $D$ of $Q(H)$ and any $\ast$-homomorphism $f: A \to M(D\otimes K)$ there is a $\ast$-homomorphism $g: A \to M(D \otimes K)$ such that $f\oplus g$ lifts to a $\ast$-homomorphism  $A\to M(\pi^{-1}(D)\otimes K)$;

  \medskip

(iii) for any separable C*-subalgebra $D$ of $Q(H)$ and any $\ast$-homomorphism $f: A \to M(D\otimes K)$ there is a $\ast$-homomorphism $g: A \to M(D \otimes K)$ such that $f\oplus g$ lifts to a ccp map  $A\to M(\pi^{-1}(D)\otimes K)$.
\end{theorem}
\begin{proof} (i) $\Rightarrow$ (ii): Let $D$ be  a separable C*-subalgebra of $Q(H)$ and   $f: A \to M(D\otimes K)$ a  $\ast$-homomorphism. Let $\{i_{\lambda}\}$ be  a quasicentral approximate unit for $D\otimes K\lhd M(D\otimes K)$. Then for any $x\in D\otimes K$ and $a\in A$
$$i_{\lambda}^{1/2}f(a)i_{\lambda}^{1/2}x \to f(a)x.$$ This means that the ccp maps $i_{\lambda}^{1/2}fi_{\lambda}^{1/2}: A \to D\otimes K$ pointwisely converge  in the strict topology to $f$. By Lemma \ref{DtensorK} each $i_{\lambda}^{1/2}fi_{\lambda}^{1/2}$ lifts.  By Theorem \ref{strict} $f$ lifts to a ccp map $\Phi: A\to M(\pi^{-1}(D)\otimes K)$. The proof proceeds now exactly as the proof of (i) $\Rightarrow$ (ii) in Theorem \ref{CharacterizationLP}.

\medskip

(ii) $\Rightarrow$ (iii) is obvious.

\medskip

(iii) $\Rightarrow$ (i): Let $\phi: A \to Q(H)$ be a ccp map. Let $D = C^*(\phi(A), 1_{Q(H)}).$  By Kasparov-Steinspring Dilation Theorem $\phi$ dilates to a $\ast$-homomorphism $f: A \to M(D\otimes K)$. There is  a $\ast$-homomorphism $g: A \to M(D \otimes K)$ such that $f\oplus g$ lifts to a ccp map  $\Psi: A\to M(\pi^{-1}(D)\otimes K)$. Then, taking into account that $D$ is unital,  $\left(\Psi_{11}\right)_{11}: A \to \pi^{-1}(D) \subset B(H)$ is a ccp lift of $\phi$.
\end{proof}

\medskip

\subsection{Free products}

\begin{corollary}\label{Boca} (Boca \cite{Boca}) The LP passes to countable free products of separable C*-algebras.
\end{corollary}
\begin{proof} Assume $A_{\alpha}$, $\alpha \in \mathbb N$,  have the LP and let $f: \ast A_{\alpha} \to M(D/I \otimes K)$ be a $\ast$-homomorphism. Let $f_{A_{\alpha}}$ be its restriction onto $A_{\alpha}$, $\alpha \in \mathbb N$. By Theorem \ref{CharacterizationLP}  there are $g_{A_{\alpha}}$, $\alpha \in \mathbb N$,  such that $f_{A_{\alpha}}\oplus g_{A_{\alpha}}$ lifts to a $\ast$-homomorphism $F_{A_{\alpha}}$. Let $F$ be the corresponding $\ast$-homomorphism on $\ast A_{\alpha}$. Then $f\oplus g = \pi \circ F$. By Theorem \ref{CharacterizationLP} $\ast A_{\alpha}$ has the LP.
\end{proof}

\begin{corollary}(Pisier \cite{PisierBook}) The LLP passes to free products of separable C*-algebras.
\end{corollary}
\begin{proof} In the same way as in the proof of Corollary \ref{Boca}, using Theorem \ref{CharacterizationLLP}  we obtain that the LLP passes to countable free products of separable C*-algebras. Now let $\ast A_{\alpha}$ be an arbitrary free product of separable C*-algebras. Let $\phi: \ast A_{\alpha} \to B/I$ be ccp and let $E\subset  \ast A_{\alpha} $ be a finite-dimensional operator system. Then there are countably many $\alpha_1, \alpha_2, \ldots, $ such that $E \subset \ast_i A_{\alpha_i} $ which implies that $\phi|_E$ is liftable.
\end{proof}

We note that Pisier proved the statement above without the assumption that C*-algebras are separable.

\medskip

\section{Free products amalgamated over finite-dimensional subalgebras}

In this section we will use several times the well known fact that if the unitary group of $D/I$ is connected, then each unitary in $D/I$ lifts to a unitary in $D$.

\subsection{LP and LLP}

\begin{lemma}\label{FinDim} Suppose $F$ is a finite-dimensional C*-algebra, $B$ is a unital C*-algebra and $\alpha, \beta: F \to B$ are unital $\ast$-homomorphisms.

\medskip

(i) If $\alpha (p)\sim \beta(p)$ for all projections $p\in F$,  then $\alpha$ and $\beta$ are unitarily equivalent;

\medskip

(ii)   If  $\alpha (p)\sim 1 \sim \beta(p)$ for all non-zero projections $p\in F$, $I$ is an ideal in $B$, $\pi: B\to B/I$ is the canonical surjection, $\pi\alpha = \pi\beta$, and the unitary group of $B/I$ is connected,  then the unitary $U$ implementing (i) can be chosen with $\pi(U) =1$.
\end{lemma}
\begin{proof} (i) Standard:  write $F$ as a direct sum of full matrix algebras and pick minimal projection $e_{11}^{(k)}$ in each summand.  Since
$\alpha(e_{11}^{(k)}) \sim \beta(e_{11}^{(k)})$, we find partial isometries $v_k \in B$ such that
\begin{equation}\label{1inLemma} v_k\alpha(e_{11}^{(k)})v_k^* = \beta(e_{11}^{(k)}).
\end{equation}
Let
$$U = \sum_k\sum_j \beta(e_{j1}^{(k)})v_k\alpha(e_{1j}^{(k)}).$$ It is straightforward to check that $U$ is unitary and for each matrix unit $e_{ij}^{(k)}$ we have
$U\alpha(e_{ij}^{(k)}) U^* = \beta(e_{ij}^{(k)})$ which implies
$$U\alpha U^* = \beta.$$

\medskip

(ii) Construct $U$ from $v_k$'s as above. Let $q_k = \pi\alpha(e_{11}^{(k)}) = \pi\beta(e_{11}^{(k)})$. Then
$$\pi(v_k)\pi(v_k)^* = q_k =  \pi(v_k)^*\pi(v_k),$$ i.e. $\pi(v_k)$ is unitary in $q_k B/I q_k$. Since $\alpha(e_{11}^{(k)}) \sim 1$, we have
$q_k \sim \pi(1) = 1$, so there is $x_k\in B/I$ such that $x_k^*x_k= 1, x_kx_k^* = q_k.$ This induces an isomorphism
$$B/I \cong q_k B/I q_k,$$ defined by $ a \mapsto x_kax_k^*$. In particular it sends 1 to $q_k$.
Therefore
$$\mathcal U(q_k B/I q_k) \cong \mathcal U(B/I) = \mathcal U_0(B/I) \cong \mathcal U_0(q_k B/I q_k).$$ So $\pi(v_k)\in q_k B/I q_k$ is connected with $q_k$ by a path of unitaries in $q_k B/I q_k$ and hence lifts to a unitary $w_k\in \beta( e_{11}^{(k)}) B \beta(e_{11}^{(k)})$.
Now in our construction of $U$ in (i) we replace $v_k$ by $w_k^*v_k$, i.e. let
$$\tilde U = \sum_k\sum_j  \beta(e_{j1}^{(k)})w_k^*v_k\alpha(e_{1j}^{(k)}).$$
%\begin{multline*}\pi(\tilde U)= \sum_k\sum_j \pi\beta(e_{j1}^{(k)})q_k \pi\alpha(e_{1j}^{(k)}) =  \sum_k\sum_j \pi\beta(e_{j1}^{(k)})q_k \pi\beta(e_{1j}^{(k)})  =\\ \sum_k\sum_j \pi\beta(e_{j1}^{(k)}e_{11}^{(k)}e_{1j}^{(k)}) = \sum_k\sum_j \pi\beta(e_{jj}^{(k)}) =1. \end{multline*}
It is straightforward to check that $\pi(\tilde U) = 1$. We have
$$w_k^*v_k \alpha(e_{11}^{(k)})(w_k^*v_k)^* = ^{(\ref{1inLemma})} w_k^*\beta(e_{11}^{(k)})w_k =  w_k^*w_kw_k^*w_k =  \beta(e_{11}^{(k)}).$$ Using this, same as in (i), it is straightforward to check that $\tilde U$ is unitary and $\tilde U\alpha\tilde U^* = \beta.$

\end{proof}

%Below we will use the following notation: if $F$ is a C*-subalgebra of $A$ and $B$ and  $\ast$-homomorphisms $f_A: A \to D$ and  $f_B: B \to D$ coincide on $F$, then we let $f_A\ast f_B$ denote  the corresponding $\ast$-homomorphism from $ A\ast_F B$ to $D$.

Below, for a $\ast$-homomorphism $r: A \to B(H)$, its composition with the canonical embedding $B(H) \hookrightarrow M(D\otimes K)$ will also be denoted by $r$.

\begin{theorem}\label{amalg} Suppose $F$ is a unital finite-dimensional C*-subalgebra of unital separable C*-algebras $A$ and $B$, $D$ is $\sigma$-unital, and  unital $\ast$-homomorphisms $f_A: A \to M(D/I \otimes K)$ and  $f_B: B \to M(D/I \otimes K)$ coincide on $F$. Suppose there exist unital $\ast$-homomorphisms  $g_A: A \to M(D/I \otimes K)$ and  $g_B: B \to M(D/I \otimes K)$ such that
$f_A\oplus g_A$ and $f_B \oplus g_B$ lift to unital $\ast$-homomorphisms to $M(D\otimes K)$. Then there exist unital $\ast$-homomorphisms $h_A: A \to M(D/I \otimes K)$ and  $h_B: B \to M(D/I \otimes K)$ such that $f_A \oplus h_A$ and $f_B\oplus h_B$ lift compatibly on $F$.
\end{theorem}
\begin{proof} Let $\sigma_A: A \to B(H)$ and $\sigma_B: B \to B(H)$ be faithful representations and let
$$r_A = \sigma_A^{(\infty)}, \; r_B = \sigma_B^{(\infty)}.$$ Let
$$\tilde g_A = g_A \oplus r_A, \; \tilde g_B = g_B \oplus r_B.$$ Then $f_A\oplus \tilde g_A$ has a lift $F_A$ of the form
$$F_A = (\text{a lift of \;} f_A\oplus g_A) \oplus r_A, $$ and $f_B\oplus \tilde g_B$ has a lift $F_B$ of the form
$$F_B = (\text{a lift of \;} f_B\oplus g_B) \oplus r_B.$$ Then for any non-zero projection $p\in F$, $F_A(p)\ge r_A(p)\sim 1, $
$F_B(p)\ge r_B(p)\sim 1. $ By \cite[6.11.7.(c)]{BlackadarK-theory}
\begin{equation}\label{EqAmalgProd1} F_A(p)\sim 1, \; F_B(p)\sim 1. \end{equation}
Similarly, for any non-zero projection $p\in F$,
\begin{equation*} \tilde g_A(p)\sim 1, \; \tilde g_B(p)\sim 1. \end{equation*}
By Lemma \ref{FinDim} (i), there is unitary $u\in M(D/I \otimes K)$ such that
$$\tilde g_B|_F = u^*\tilde g_A|_F u.$$
Then $f_A \oplus u^*\tilde g_A u$ and $f_B \oplus \tilde g_B$ agree on $F$ and both lift. Indeed, since the unitary group of $M(D/I\otimes K)$ is connected (\cite{CuntzHigson}, see also \cite[Note added in proof, p. 410]{Mingo}), $u$ has a unitary lift $U$, and then  $f_A \oplus u^*\tilde g_A u$ lifts to $\tilde F_A = (1\oplus U)^*F_A(1\oplus U). $
 Moreover, it follows from (\ref{EqAmalgProd1}) that
for any non-zero  projection $p\in F$,
$$\tilde F_A(p) = (1\oplus U)^*F_A(p)(1\oplus U) \sim 1.$$  Since the unitary group of $M(D/I\otimes K)$ is connected, by Lemma \ref{FinDim} (ii), there is $U_1\in M(D\otimes K)$ such that $\pi(U_1) = 1$ and $\tilde F_A|_F = U_1^*F_B|_FU_1.$
Then $U_1^*F_BU_1$ is a lift of $f_B\oplus \tilde g_B$ and agrees on $F$ with the lift $\tilde F_A$ of $f_A \oplus u^*\tilde g_Au$.
\end{proof}

In the rest of this section  $F$  is a finite-dimensional $C^*$-subalgebra of $A$ and $B$.

\begin{corollary}\label{LPpassesToFreeProducts} If separable C*-algebras $A$ and $B$ have the LP, then  $A\ast_F B$ has the LP.
\end{corollary}
\begin{proof} If $A$, $B$, and the amalgamation are unital, then the statement follows from  Theorem \ref{amalg} and Theorem \ref{CharacterizationLP}. The non-unital case follows from the following lemma.
\end{proof}

\begin{lemma}\label{ForcedUnitization} Let $A^+$ denote the forced unitization of $A$. Then for any $A, B, C$
$$A^+\ast_{C^+} B^+ = (A\ast_C B)^+.$$
\end{lemma}
\begin{proof} It is straightforward to verify the statement by using the universal property of forced unitization (it uniquely extends homomorphisms to unital C*-algebras to unital homomorphisms).
\end{proof}

\begin{corollary} (Ozawa \cite{Ozawa}) If separable C*-algebras $A$ and $B$ have the LLP, then  $A\ast_F B$ has the LLP.
\end{corollary}
\begin{proof} By Theorem \ref{CharacterizationLLP}, Theorem \ref{amalg} and Lemma \ref{ForcedUnitization}.
\end{proof}

 We note that Ozawa proved the statement above without the assumption of separability.

\medskip

In \cite{ESS} it was proved that for any finitely generated virtually free group $G$, $C^*(G)$ is semiprojective. The same strategy and the use of Corollary \ref{LPpassesToFreeProducts} shows

\begin{corollary}\label{VirtuallyFree} Let $G$ be a finite tree product with finite edge groups and countable vertex groups. For a vertex $v$, let $G_v$ denote the corresponding vertex
group. If for each vertex $v$, $C^*(G_v)$ has the LP,  then $C^*(G)$ has the LP. In particular, if $G$ is a finitely generated virtually free group, then $C^*(G)$ has the LP.
\end{corollary}

%\begin{corollary} Let $G$ be a finite graph of groups with finite edge groups. If for each vertex $v$, $C^*(G_v)$ has the LP,  then the fundamental group of $G$ has the LP. \end{corollary}

\medskip

\subsection{Some other lifting properties}

The same technique as above works for some other lifting properties. We note that in Lemma \ref{FinDim} (ii)
the assumptions that $\alpha (p)\sim 1$ and $ \beta(p)\sim 1$ and that the unitary group of $B/I$ is connected are needed only
to ensure that for a partial isometry $v_p$ conjugating $\alpha(p)$ and $\beta(p)$, the unitary $\pi(v_p)\in \pi(\beta(p)B\beta(p))$ lifts to a unitary in $\beta(p)B\beta(p)$.
Thus Lemma \ref{FinDim} (ii) can be reformulated in the following way.

\begin{lemma}\label{FinDimTechnical} Suppose $F = \sum_{i=1}^m M_{n_i}$ is a finite-dimensional C*-algebra, $B$ is a unital C*-algebra, $I$ is an ideal in $B$, $\pi: B\to B/I$ is the canonical surjection, $\alpha, \beta: F \to B$ are unital $\ast$-homomorphisms, and  $\pi\alpha = \pi\beta$.
For each $i=1, \ldots, m$, let $p_i\in  M_{n_i}$ be a minimal projection.
 If, for each  $i=1, \ldots, m$, $\alpha (p_i)\sim \beta(p_i)$, and there is a partial isometry $v_i$ conjugating $\alpha (p_i)$ and $ \beta(p_i)$ such that $\pi(v_i)$ lifts to a unitary in $\beta(p_i)B\beta(p_i)$, then   there is a unitary $U\in B$ such that
 $\alpha = U^*\beta U$ and  $\pi(U) =1$.
\end{lemma}

Recall that a separable C*-algebra $A$ is {\it semiprojective}  if for every separable C*-algebra $B$, every increasing sequence of ideals $J_1\subseteq J_2 \subseteq \ldots $ in $B$, and every $\ast$-homomorphism $\phi: A \to B/{\overline {\bigcup J_k}}$, there exists an $n\in \mathbb N$ and a $\ast$-homomorphism $\psi:A \to B/J_n$ such that
$\pi_{n, \infty}\circ \psi = \phi$,
where $\pi_{n, \infty}: B/J_n \to B/\overline{\bigcup J_k}$  is the natural quotient map.

In  \cite{BlackadarSP} Blackadar states that the class  of semiprojective C*-algebras is closed under free products amalgamated over finite-dimensional subalgebras. For a proof he states that it is similar to Effros and Kaminker's proof of analogous result for their version of semiprojectivity \cite{EffrosKaminker}. Here we give an alternative and explicit proof of this fact.

\begin{corollary}\label{SP} (Blackadar \cite{BlackadarSP}) If $A$ and $B$ are semiprojective, then $A\ast_F B$ is semiprojective.
\end{corollary}
\begin{proof} The non-unital case will follow from the unital one by Lemma \ref{ForcedUnitization} and the fact that a C*-algebra is semiprojective if and only if its forced unitization is semiprojective. So suppose unital $\ast$-homomorphisms $f_A: A \to D/\overline{\bigcup I_n}$ and $f_B: B \to D/\overline{\bigcup I_n}$ agree on $F$. There is $N\in \mathbb N$ such that they lift to unital $F_A: A \to D/ I_N$ and $F_B: B \to D/ I_N$ respectively.
Let $p_i\in  M_{n_i}$ be a minimal projection, $i=1, \ldots, m$.
It is known (\cite{BlackadarSP}) that in quotients of the form $D/\overline{\bigcup I_n}$ any unitary and any two equivalent projections lift to $D/I_N$, for $N$ sufficiently large.
Therefore  we can assume that  for each  $i=1, \ldots, m$, $F_A (p_i)\sim F_B(p_i)$, and there is a partial isometry $v_i$ conjugating them such that the unitary
$\pi(v_i)\in f_B(p_i)\left(D/\overline{\bigcup I_n}\right)f_B(p_i)  $ lifts to a unitary in $ F_B(p_i)\left(D/I_N\right) F_B(p_i)$.
By Lemma \ref{FinDimTechnical} there is $U\in D/I_N$ such that $\pi(U) =1$ and
 $$F_B|_F = U^* F_A|_F U.$$ Then  $U^* F_A U$ and $F_B$ are lifts of $f_A, f_B$ respectively and agree on $F$.
\end{proof}

\medskip

Let $M_n$ be the algebra of $n$-by-$n$ matrices. A C*-algebra is called {\it residually finite-dimensional} (RFD)
%if it embeds into $\prod_{n=1}^{\infty} M_n$ or, in other words,
if it has a separating family of finite-dimensional representations.

\medskip

%Let $H= l^2(\mathbb N)$. We will identify $M_n$ with $$B(l^2\{1, \ldots, n\}) \subset B(H).$$
Let $H$ be a Hilbert space and let $P_{\alpha}$, $\alpha \in \Lambda$,  be an increasing net of projections of dimension $n_{\alpha}$ $\ast$-strongly converging to $1_{B(H)}$. We will identify $M_{n_{\alpha}}$ with $P_{\alpha} B(H) P_{\alpha}$.
Let $\mathcal D \subset \prod_{\alpha\in \Lambda} M_{n_{\alpha}}$ be the C*-algebra of all $\ast$-strongly convergent nets of operators indexed by $\Lambda$ and let $\pi: \mathcal D \to B(H)$ be the surjection that sends each net  to its $\ast$-strong limit.

%A C*-algebra is RFD if and only if ..... [].

%If $f_A: A \to D$ and $f_B: B \to D$ agree on $F$, the corresponding homomorphism from $A\ast_F B\to D$ will be denoted by $f_A\ast f_B$.

\begin{corollary}\label{RFD}(Li and Shen in separable case \cite{LiShen})\label{LiShen} Let $A$ be unital. Then $A\ast_F A$ is RFD if and only if $A$ is RFD.
\end{corollary}
\begin{proof} "Only if" is obvious.

"If":  Let $f: A\ast_F A \to B(H)$ be an embedding.  Let $f_i$ be the restriction of $f$ on the $i$-th copy of $A$, $i=1, 2$. We will use notation $f= f_1\ast f_2$.
Then $(f_1\oplus f_2)\ast (f_2\oplus f_1) = (f_1\ast f_2) \oplus (f_2\ast f_1)$ is also an embedding. By \cite[proof of Th. 3.2]{ExelLoring}, since $A$ is RFD, there is an increasing net of projections $P_{\alpha}$ $\ast$-strongly converging to $1_{B(H)}$ and nets $\sigma_{\alpha}^{(1)}, \sigma_{\alpha}^{(2)}$ of representations of $A$ living on $P_{\alpha}H$ such that
$$ \sigma_{\alpha}^{(i)}(a) \to f_i(a) \;\; {\ast}-\text{strongly},$$ for $i=1, 2$ and $a\in A$. In other words,
$f_1$ and $f_2$ lift to $\ast$-homomorphisms $\tilde f_1, \tilde f_2: A \to \mathcal D$ respectively.
Then  $f_1\oplus f_2$ and $f_2\oplus f_1$ lift to $\bar f_1 = \tilde f_1\oplus \tilde f_2$ and $\bar f_2 = \tilde f_2\oplus \tilde f_1$, and for any projection $p\in F$, $\bar f_1(p) \sim \bar f_2(p)$.
Since $\pi(\bar f_2(p)) B(H) \pi(\bar f_2(p))$ has connected unitary group, any unitary in it lifts. By Lemma \ref{FinDimTechnical} there is $U\in \mathcal D$ such that $\pi(U) =1$ and
 $$\bar f_2|_F = U^* \bar f_1|_F U.$$ Then  $U^* \bar f_1 U$ and $\bar f_2$ are lifts of $f_1\oplus f_2$ and $f_2\oplus f_1$ respectively and agree on $F$.  Thus $\left(U^* \bar f_1 U\right)\ast \bar f_2$ is a lift of the embedding $(f_1\oplus f_2)\ast (f_2\oplus f_1)$ and hence is an embedding.
\end{proof}

We will deduce the general result from the result above. For that we prove the following statement.

\begin{theorem}\label{EmbeddingAmalgamated} Let $A, B, D$  and $C$ be unital separable C*-algebras. Let $i_A: C \to A$, $i_B: C \to B$,  and $\phi_A: A \to D$, $\phi_B: B \to D$ be unital inclusions such that $\phi_A\circ i_A= \phi_B\circ i_B$. Then $A\ast_C B$ embeds into $D\ast_C D$.
\end{theorem}
\begin{proof} Since $\phi_A\circ i_A= \phi_B\circ i_B$, the $\ast$-homomorphism $\phi_A \ast \phi_B: A \ast_C B \to D\ast_C D$ that sends $A$ to the first copy of $D$ via $\phi_A$ and $B$ to the second copy of $D$ via $\phi_B$, is well-defined. We will prove that it is injective.

Let $\alpha: A \ast_C B \to B(H)$ be an embedding such that its composition with the quotient map $q: B(H) \to Q(H)$ is still injective (for instance take an embedding $j$ and let $\alpha := j^{\oplus \infty}$). Let $\alpha_A = \alpha|_A, \alpha_B = \alpha|_B$. We have
$$\alpha_A\circ i_A = \alpha_B\circ i_B.$$ By Arveson Extension Theorem, $\alpha_A$ extends to a ccp map $\bar\alpha_A$ on $D$. By Steinspring Dilation Theorem, $\bar \alpha_A$ can be dilated to a $\ast$-homomorphism $\sigma_A$.

$$\begin{tikzcd} C \arrow{r}{i_A} & A\arrow{r}{\alpha_A}\arrow[swap]{d}{\phi_A} & B(H) \\
 & D \arrow[dashed, swap]{r}{\sigma_A} \arrow[dashed]{ur}{\bar\alpha_A} & B(H\oplus H') \arrow[swap]{u}{(.)_{11} }
\end{tikzcd}$$

\noindent Note that since $\sigma_A\circ \phi_A \circ i_A$ is a $\ast$-homomorphism and $\left(\sigma_A\circ \phi_A \circ i_A\right)_{11} = \bar\alpha_A\circ \phi_A \circ i_A = \alpha_A\circ i_A$ is a $\ast$-homomorphism, $\left(\sigma_A\circ \phi_A \circ i_A\right)_{22}$ is a $\ast$-homomorphism as well.

We obtain $\bar \alpha_B, \sigma_B$ in the same way.

Consider the representations $\left(\sigma_A\circ \phi_A\right)^{\oplus\infty}: A \to B(H^{\oplus\infty}),$
$\left(\sigma_B\circ \phi_B\right)^{\oplus\infty}: B \to B(H^{\oplus\infty})$.  When restricted to $C$ these maps do not agree (yet) but
\begin{multline*} \left(\sigma_A\circ\phi_A\circ i_A\right)^{\oplus\infty} = \left(\sigma_A\circ\phi_A\circ i_A\right) \oplus \left(\sigma_A\circ\phi_A\circ i_A\right)^{\oplus\infty}\\ = \alpha_A\circ i_A \oplus \left(\sigma_A\circ\phi_A\circ i_A\right)_{22}
\oplus \left(\sigma_A\circ\phi_A\circ i_A\right)^{\oplus\infty}, \end{multline*}
\begin{multline*} \left(\sigma_B\circ\phi_B\circ i_B\right)^{\oplus\infty}  = \alpha_B\circ i_B \oplus \left(\sigma_B\circ\phi_B\circ i_B\right)_{22}
\oplus \left(\sigma_B\circ\phi_B\circ i_B\right)^{\oplus\infty}. \end{multline*}
The $(\;)_{11}$-components agree, and for the "remainders", $\left(\sigma_A\circ\phi_A\circ i_A\right)_{22}
\oplus \left(\sigma_A\circ\phi_A\circ i_A\right)^{\oplus\infty}$ and $ \left(\sigma_B\circ\phi_B\circ i_B\right)_{22}
\oplus \left(\sigma_B\circ\phi_B\circ i_B\right)^{\oplus\infty}$, we have
\begin{multline*} rank \left(\left(\sigma_A\circ\phi_A\circ i_A\right)_{22}
\oplus \left(\sigma_A\circ\phi_A\circ i_A\right)^{\oplus\infty}\right)(c) = \infty \\
=rank \left(\left(\sigma_B\circ\phi_B\circ i_B\right)_{22}
\oplus \left(\sigma_B\circ\phi_B\circ i_B\right)^{\oplus\infty}\right)(c),\end{multline*} for any $0\neq c\in C$. Therefore, by Voiculescu's theorem, there exists unitary $u\in B(H^{\oplus\infty})$ such that
\begin{multline*} \left(\sigma_A\circ\phi_A\circ i_A\right)_{22}
\oplus \left(\sigma_A\circ\phi_A\circ i_A\right)^{\oplus\infty} =  u^* \left(\left(\sigma_B\circ\phi_B\circ i_B\right)_{22}
\oplus \left(\sigma_B\circ\phi_B\circ i_B\right)^{\oplus\infty}\right) u \\ \mod K(H).\end{multline*}
Then the $\ast$-homomorphisms $q\circ \left((1_{B(H)}\oplus u)\sigma_A^{\oplus\infty}(1_{B(H)}\oplus u)^*\right)$
and $q\circ \sigma_B^{\oplus\infty}$ agree on $\phi_A\circ i_A(C)$. Therefore
$$\gamma:=  \left(q\circ \left((1_{B(H)}\oplus u)\sigma_A^{\oplus\infty}(1_{B(H)}\oplus u)^*\right)\right)\ast \left(q\circ \sigma_B^{\oplus\infty}\right): D\ast_C D \to Q(H^{\oplus\infty})$$
is well-defined and the diagram

$$\begin{tikzcd} A\ast_C B \arrow{r}{\alpha} \arrow{d}{\phi_A\ast \phi_B} &  B(H) \arrow{r}{q} &Q(H) \\
  D\ast_C D \arrow{rr}{\gamma} & & Q((H\oplus H')^{\oplus\infty})\arrow[swap]{u}{((.)_{11})_{11} }
\end{tikzcd}$$
is commutative. Since $q\circ\alpha$ is injective, so must be $\phi_A\ast\phi_B$.
\end{proof}

%\begin{remark} The assumption of separability of $C$ in Theorem \ref{EmbeddingAmalgamated} can be removed. For that we replace $(\sigma_A\phi_Ai_A)^{\oplus\infty}$ and  $(\sigma_B\phi_Bi_B)^{\oplus\infty}$  by $(\sigma_A\phi_Ai_A)^{\oplus\infty} \oplus (\sigma_B\phi_Bi_B)^{\oplus\infty}$ and $(\sigma_B\phi_Bi_B)^{\oplus\infty} \oplus (\sigma_A\phi_Ai_A)^{\oplus\infty}$ respectively to ensure that ranks of the "remainders" are infinite of the same cardinality, and then use the non-separable version of Voiculescu's theorem \cite{HadwinNonseparableVoiculescu}.\end{remark}

\begin{theorem}\label{RFDunital}(Li and Shen \cite{LiShen}) Let $A$ and $B$ be unital separable\footnote{Li and Shen's proof uses the separability hypothesis, although it is not included in the statement of their theorem.} C*-algebras, $F$ a finite-dimensional C*-algebra and $i_A: F \to A$, $i_B: F \to B$ unital inclusions. Then the corresponding unital amalgamated free product $A\ast_F B$ is RFD if and only if there exist a sequence $(n_k)$ and  unital inclusions $\phi_A: A\to \prod M_{n_k}, \phi_B: B \to \prod M_{n_k}$ such that $\phi_A\circ i_A = \phi_B\circ i_B$.
\end{theorem}
\begin{proof} The "only if" is obvious.

\noindent "If": Let $\phi_A: A\to \prod M_{n_K}, \phi_B: B \to \prod M_{n_k}$  be unital inclusions such that $\phi_A\circ i_A = \phi_B\circ i_B$. Let $D= C^*(\phi_A(A)\cup \phi_B(B)) \subset \prod M_{n_k}.$ Then $D$ is separable and RFD. 
By Theorem \ref{EmbeddingAmalgamated}, $A\ast_F B$ embeds into $D\ast_F D$  which is RFD by Corollary \ref{RFD}. Therefore $A\ast_F B$ is RFD.
\end{proof}

Using Lemma \ref{ForcedUnitization} we also obtain the non-unital case.

\begin{theorem}\label{RFDnonunital} Let $A$ and $B$ be separable C*-algebras, $F$ a finite-dimensional C*-algebra and $i_A: F \to A$, $i_B: F \to B$ inclusions. Then the corresponding amalgamated free product $A\ast_F B$ is RFD if and only if there exist inclusions $\phi_A: A\to \prod M_n, \phi_B: B \to \prod M_n$ such that $\phi_A\circ i_A = \phi_B\circ i_B$.
\end{theorem}

\section{Soft tori have the LP}

Fix $\epsilon \ge 0$. Recall  that the {\it soft torus} corresponding to  $\epsilon$ is the universal C*-algebra
$$C(\mathbb T^2)_{\epsilon} = C^*\langle u, v\;|\; u \; \text{and} \; v \; \text{are unitaries and} \; \|[u, v]\| \le \epsilon\rangle.$$

When $\epsilon = 0$ one has $C(\mathbb T^2)_{\epsilon} = C(\mathbb T^2)$  the C*-algebra of all continuous functions on the torus, and when $\epsilon \ge 2$ one has
$C(\mathbb T^2)_{\epsilon} =  C^*(F_2)$ the full C*-algebra  of the free group with two generators.

\medskip

\begin{lemma}\label{commutator} If $w$ is unitary and $x$ is arbitrary, then $\|[w, x]\| = \|[w^*, x]\|$.
\end{lemma}
\begin{proof} Because $[w, x] = - w[w^*, x]w$.
\end{proof}

\medskip

Let $f: C(\mathbb T^2)_{\epsilon} \to B$ be a unital $\ast$-homomorphism. The following construction will be used several times.

\medskip

 \begin{construction}\label{construction} Let $U = f(u); V = f(v)$. Using Lemma \ref{commutator} we can define a $\ast$-homomorphism  $g: C(\mathbb T^2)_\epsilon \to B$ by $$g(u) =-U^*, \;\;g(v) = V.$$ For each $t\in [0,1]$ let
 $$U_t = \left(\begin{array}{cc} tU & \sqrt{1-t^2}\;\mathbb 1\\\sqrt{1-t^2}\;\mathbb 1 & -tU^*\end{array}\right),\;\;
 V_t = \left(\begin{array}{cc} V&0\\0&V\end{array}\right). $$ Using Lemma \ref{commutator} we obtain
 \begin{equation}\label{SmallerCommutator}\|[U_t, V_t]\| = \left\|\left(\begin{array}{cc} t[U, V]&0\\0&t[V, U^*]\end{array}\right)\right\|\le t\epsilon.\end{equation}
  Let $f_t: C(\mathbb T^2)_\epsilon \to M_2(B)$ be the  $\ast$-homomorphism defined by the pair $U_t, V_t$. We have
  \begin{equation}\label{eq0}  f_t \to f\oplus g \;\text{ pointwisely when}\; t\to 1.\end{equation}
\end{construction}

  \medskip
\begin{lemma}\label{FormerClaim} Let $\epsilon > 0$. Let $f: C(\mathbb T^2)_\epsilon \to B/I$ be  a unital $\ast$-homomorphism and  $f_t$ be defined as above. If $f(v)$  lifts to a unitary in $B$, then for each $0<t < 1$, $f_t$ lifts to a $\ast$-homomorphism.
\end{lemma}
\begin{proof}
    Fix  $0< t<1$. Let $\tilde V$ be a unitary lift of $V = f(v)$ and let $X$ be any contractive lift of $U = f(u)$.
  Let $\{e_{\lambda}\}$ be a quasicentral approximate unit for $I \triangleleft B$.
  Choose $t'$ such that $t<t'<1$.
  Since $$\pi([XX^*, \tilde V]) = \pi([X^*X, \tilde V]) =  [\mathbb 1, V] =0, \;\text{and}\; \|\pi([X, \tilde V])\| \le \epsilon,$$ there is $\lambda$ such that for $Y= (1-e_{\lambda})X$ we have

  \begin{equation}\label{eq2} \|[YY^*, \tilde V]\|\le \left(\frac{4(1-t')\epsilon}{5}\right)^2, \end{equation}

  \begin{equation}\label{eq3} \|[Y^*Y, \tilde V]\|\le \left(\frac{4(1-t')\epsilon}{5}\right)^2, \end{equation}

    \begin{equation}\label{eq1} \|[Y, \tilde V]\|\le \frac{\epsilon t'}{t} \;\text{and}\; \|[Y^*, \tilde V]\|\le \frac{\epsilon t'}{t}. \end{equation}

  Let

  $$\tilde U_t = \left(\begin{array}{cc} tY&(1-t^2YY^*)^{1/2} \\ (1-t^2Y^*Y)^{1/2} & -tY^*\end{array}\right), \;\;\; \tilde V_t
  = \left(\begin{array}{cc} \tilde V & 0\\0 & \tilde V\end{array}\right).$$

  These are unitary lifts of $U_t, V_t$ respectively. We have

  $$ [\tilde U_t, \tilde V_t] = \left(\begin{array}{cc} t[Y, \tilde V] & [ \left(1- t^2 YY^*\right)^{1/2}, \tilde V] \\
  \left[\left(1- t^2 Y^*Y\right)^{1/2}, \tilde V\right] &
  -t[Y^*, \tilde V] \end{array}\right).$$

  By (\ref{eq2}) and Pedersen's inequality (which says that $\|[A^{1/2}, B]\|\le \frac{5}{4}\|[A, B]\|^{1/2}$ when $0\le A \le \mathbb 1$ and $\|B\|\le 1$ ) we have

  \begin{equation}\label{eq4} \left \|[\left(1-t^2YY^*\right)^{1/2}, \tilde V]\right\|\le \frac{5}{4}\left \|t^2[YY^*, \tilde V]\right \|^{1/2} \le \frac{5}{4}
  \left\|[YY^*, \tilde V]\right\|^{1/2} \le (1-t')\epsilon \end{equation}
  and the same for $\left\|[\left(1-t^2Y^*Y\right)^{1/2}, \tilde V]\right\|$. By (\ref{eq1}) and (\ref{eq4}) we obtain

  \begin{multline*}\|[\tilde U_t, \tilde V_t]\|\le  \left\|\left(\begin{array}{cc} t[Y, \tilde V] & 0\\
  0 &
  -t[Y^*, \tilde V] \end{array}\right)\right\| + \\ \left\|\left(\begin{array}{cc}  0& [ \left(1- t^2 YY^*\right)^{1/2}, \tilde V] \\
  \left[\left(1- t^2 Y^*Y\right)^{1/2}, \tilde V\right] & 0
   \end{array}\right) \right\| \le t'\epsilon + (1-t')\epsilon = \epsilon.\end{multline*}

  Hence $\tilde U_t, \tilde V_t$ define a $\ast$-homomorphism $\tilde f_t$ which is a lift of $f_t$.

  \end{proof}

\begin{remark}\label{RemarkCommutators} The proof of the Lemma \ref{FormerClaim} shows that for a  finite collection of unitaries $U_i, V_i$ in $B/I$ with  $\|[U_i, V_j]\|\le \epsilon$,   given a unitary lift $\tilde V_i$ of $V_i$ for each $i$,  one can construct a unitary lift
$\tilde U_{t, i}$ of $\left(\begin{array}{cc} tU_i & \sqrt{1-t^2}\;\mathbb 1\\\sqrt{1-t^2}\;\mathbb 1 & -tU_i^*\end{array}\right)$  so that
$$\left\| \left[\tilde U_{t, i},  \left(\begin{array}{cc} \tilde V_j&0\\0&\tilde V_j\end{array}\right)\right]\right\|\le \epsilon, $$ for all $i, j$.

\end{remark}

\medskip

  \begin{theorem}\label{SoftTorusthe LP} Soft torus has the LP.
 \end{theorem}
 \begin{proof} If $\epsilon =0$, the soft torus is $C(\mathbb T^2)$ which is a nuclear C*-algebra and therefore has the LP. Assume $\epsilon >0$.  Let  $f: C(\mathbb T^2)_\epsilon \to  M(B/I \otimes K)$ be a unital $\ast$-homomorphism.  Let $U = f(u), V=f(v)$. By \cite{CuntzHigson} (see also \cite[Note added in proof, p. 410]{Mingo}), $V$ (and $U$) is in the connected component of $1$ in $M(B/I \otimes K)$ and hence lifts to a unitary in $M(B \otimes K)$. Let $g$ be as in Construction \ref{construction}.
  By Lemma \ref{FormerClaim} and  (\ref{eq0}),   $f\oplus g$ is a pointwise limit of liftable $\ast$-homomorphisms $f_t$ as $t\to 1$. By Arveson's result that a pointwise limit of liftable cp maps is liftable, we conclude that $f\oplus g$ lifts to a cp map. By Theorem \ref{CharacterizationLP}, $C(\mathbb T^2)_\epsilon$ has the LP.
  \end{proof}

\begin{theorem}\label{FnxFn}  For any $n\in \mathbb N \bigcup \{\infty\}$, $C^*(F_n\times F_n)$ is an inductive limit  of C*-algebras with the LP.
\end{theorem}
\begin{proof}  Let at first $n\in \mathbb N$. The C*-algebra $C^*(F_n\times F_n)$ can be written as the inductive limit (with surjective connecting maps) of the universal C*-algebras
$$A_\epsilon = C^*\langle u_i, v_i \;|\;  u_i \; \text{and} \; v_i \; \text{are unitaries}, \;  i= 1, \ldots, n, \; \|[u_i, v_j]\| \le \epsilon,\forall  i, j = 1, \ldots, n \rangle$$

\noindent as $\epsilon \to 0$. So it is sufficient to show that each $A_{\epsilon}$ has the LP.  By Remark \ref{RemarkCommutators}, the proof proceeds exactly as in Theorem \ref{SoftTorusthe LP}.

 Since $C^*(F_{\infty}\times F_{\infty})$ can be written as an inductive limit of $C^*(F_n\times F_n)$ as $n\to \infty$ (with injective connecting maps), combining this with the result for $n\in \mathbb N$ one obtains easily that $C^*(F_{\infty}\times F_{\infty})$ is an inductive limit of C*-algebras with the LP.

\end{proof}

An inductive limit (with surjective connecting maps) of C*-algebras with the LP need not have the (L)LP. Such an example is constructed in \cite[Ex. 2.16]{Ozawa}. Below we give
examples of different kind showing this.

\begin{proposition} Let $A$ be any separable C*-algebra without the (L)LP. Then the cone $CA$  is an inductive limit of C*-algebras with the LP but does not have the (L)LP.
\end{proposition}
\begin{proof} The cone over any separable C*-algebra is an inductive limit of C*-algebras that are projective (\cite{LoringShulman}) and therefore have the LP. It remains to show that if $A$ fails the (L)LP, then so does $CA$. Let $ev_1: CA \to A$ be the evaluation at 1 and let $\delta: A \to CA$ be given by $\delta(a) = t \otimes a$. If $CA$ had the (L)LP, then
for any cp map $\phi$ from $A$ to any quotient $B/I$ (the Calkin algebra, respectively), the map $\phi\circ ev_1$ would lift to a cp map $\psi$ from  CA  to $B$ ($B(H)$, respectively).  It is straightforward to check that then $\psi\circ \delta$ would be  a cp lift of $\phi$.
\end{proof}

\medskip

\section{Other properties of soft tori}

Here we give short proofs of several known results  about soft tori and prove some new properties of them.

\subsection{RFD}

\begin{theorem}\label{SoftTorusRFD}(Eilers-Exel \cite{EilersExel}) Soft torus is RFD.
\end{theorem}
\begin{proof} When $\epsilon =0$, teh soft torus is $C(\mathbb T^2)$ which is RFD. So we can assume that $\epsilon >0$. Let $\mathcal D$ be the C*-algebra of all $\ast$-strongly converging sequences of matrices (this C*-algebra was already used in Corollary \ref{LiShen}). Then $B(H)$ is its quotient.
Let $f$ be an embedding of soft torus into $B(H)$. Let $f_t$ and $g$ be as in the Construction \ref{construction}. Since $f_t\to f\oplus g$ pointwisely as $t\to 1$, the family $\{f_t\}_{t<1}$ separates points of $C(\mathbb T^2)_{\epsilon}$.  Since any unitary $U$ in $B(H)$ lifts to a unitary  in $\mathcal D$ (indeed, just write $U = e^{iS}$ and lift $S$), by Lemma \ref{FormerClaim} for each  $t<1$, $f_t$  lifts to a sequence of finite-dimensional representations $\tilde f_t^{(i)}$. Then the family $\{\tilde f_t^{(i)}\}_{i\in \mathbb N, t<1}$ separates points of $C(\mathbb T^2)_{\epsilon}$.
\end{proof}

\medskip

\subsection{Continuous field}

For any $\epsilon>0$ and any $\epsilon_1> \epsilon_2> 0$ let $$\phi_{\epsilon}: C^*(F_2)\to C(\mathbb T^2)_{\epsilon}\;\; \text{ and}\;\; \phi_{\epsilon_1, \epsilon_2}:
C(\mathbb T^2)_{\epsilon_1} \to C(\mathbb T^2)_{\epsilon_2}$$ be the  canonical surjections. Clearly $\phi_{\epsilon_1, \epsilon_2} \circ  \phi_{\epsilon_1} = \phi_{\epsilon_2}.$

One of the main results of \cite{Exel} states  that soft tori form a continuous field of C*-algebras over the interval $[0,2]$ such that
$C^*(\mathbb T^2)_{\epsilon}$ is the fiber over $\epsilon$. Here we give a short proof of this fact.  Exel obtains his result by proving that
$$\epsilon \mapsto \|\phi_{\epsilon}(a)\|$$ is a continuous function on $[0, 2]$, for any $a\in C^*(F_2)$. Since this is an increasing function,
for any $\epsilon_0$ the limits $$\lim_{\epsilon \to \epsilon_0^{-}} \|\phi_{\epsilon}(a)\|\;\; \text{ and} \;\; \lim_{\epsilon \to \epsilon_0^{+}} \|\phi_{\epsilon}(a)\|$$ exist and therefore to prove continuity at $\epsilon_0$ it is sufficient  to show
$$\lim_{\epsilon \to \epsilon_0^{-}} \|\phi_{\epsilon}(a)\| =  \|\phi_{\epsilon_0}(a)\|\;\; \text{ and } \;\; \lim_{\epsilon \to \epsilon_0^{+}} \|\phi_{\epsilon}(a)\| =  \|\phi_{\epsilon_0}(a)\|.$$
 As was mentioned in \cite{Exel},  the hard part is the first equality.

We start with the non-trivial part.

\begin{lemma}\label{Lemma1ContinuousField} $\lim_{\epsilon \to \epsilon_0^{-}} \|\phi_{\epsilon}(a)\| = \|\phi_{\epsilon_0}(a)\|$.
\end{lemma}
\begin{proof} Since $\|\phi_{\epsilon}(a)\|$ is an increasing function of $\epsilon$, we only need to show that
$\lim_{\epsilon \to \epsilon_0^{-}} \|\phi_{\epsilon}(a)\| \ge \|\phi_{\epsilon_0}(a)\|$. Let $\pi: C(\mathbb T^2)_{\epsilon_0} \to B(H)$ be a representation such that $\|\pi(\phi_{\epsilon_0}(a))\| = \|\phi_{\epsilon_0}(a)\|.$ We define a representation  $\sigma: C(\mathbb T^2)_{\epsilon_0} \to B(H)$ by
$$\sigma(u) = \pi(u), \;\; \sigma(v) = - \pi(v^*).$$ By Construction \ref{construction}, $\pi\oplus \sigma$ is a pointwise limit of representations $\rho_n$ of $C(\mathbb T^2)_{\epsilon}$ with $\|[\rho_n(u), \rho_n(v)]\|\le t_n \epsilon_0$, where $t_n< 1$ for each $n$,  and $t_n\to 1$. Let $\epsilon_n = t_n\epsilon_0$, $n\in \mathbb N$. Then
$\epsilon_n < \epsilon_0$, $\epsilon_n\to \epsilon _0$, and $\rho_n$ factors through $C(\mathbb T^2)_{\epsilon_n}$, that is can be written as $\rho_n = \gamma_n \circ \phi_{\epsilon_0, \epsilon_n}$, for some representation $\gamma_n$ of $C(\mathbb T^2)_{\epsilon_n}$. Then
\begin{multline*} \|\phi_{\epsilon_0}(a)\| = \|\pi(\phi_{\epsilon_0}(a))\| \le \|(\pi\oplus \sigma)(\phi_{\epsilon_0}(a))\|  = \lim \|\rho_n(\phi_{\epsilon_0}(a))\| \\ = \lim \|\gamma_n(\phi_{\epsilon_0, \epsilon_n}(\phi_{\epsilon_0}(a)))\|  = \lim \|\gamma_n(\phi_{\epsilon_n}(a))\|
\le \limsup \|\phi_{\epsilon_n}(a)\| = \lim_{\epsilon\to \epsilon_0^{-}} \|\phi_{\epsilon}(a)\|.
\end{multline*}
\end{proof}

The next lemma is a proof of the easy part. It holds not only for soft tori but for any "softening" of relations  of any C*-algebra.

\begin{lemma}\label{Lemma2ContinuousField} $\|\phi_{\epsilon_0}(a)\|= \lim_{\epsilon \to \epsilon_0^+} \|\phi_{\epsilon}(a)\|.$
\end{lemma}
\begin{proof}
Let $\epsilon_n$, $n\in \mathbb N$,  be such that $\epsilon_n > \epsilon_0$ and $\epsilon_n \to \epsilon_0$. Let $$q: \prod C(\mathbb T^2)_{\epsilon_n} \to
\prod C(\mathbb T^2)_{\epsilon_n} / \bigoplus C(\mathbb T^2)_{\epsilon_n}$$ be the canonical surjection. Since $\epsilon_n \to \epsilon_0$, the $\ast$-homomorphism
$$j: C(\mathbb T^2)_{\epsilon_0}  \to \prod C(\mathbb T^2)_{\epsilon_n} / \oplus C(\mathbb T^2)_{\epsilon_n},  \;\;\; \phi_{\epsilon_0}(a) \mapsto  q\left(\prod \phi_{\epsilon_n}(a)\right)$$ is well-defined. Then
$$\|\phi_{\epsilon_0}(a)\|\ge \|j(\phi_{\epsilon_0}(a))\| = \limsup \|\phi_{\epsilon_n}(a)\| = \lim_{\epsilon \to \epsilon_0^+} \|\phi_{\epsilon}(a)\|$$ and, since $\|\phi_{\epsilon}(a)\|$ is an increasing function, the result follows.
\end{proof}

\medskip

\begin{theorem}\label{ContinuousField} (Exel \cite{Exel}) There exists a continuous field of C*-algebras over the interval $[0,2]$ such
that $C(\mathbb T^2)_{\epsilon}$ is the fiber over $\epsilon$ and such that
$\epsilon\in [0, 2] \mapsto  \phi_{\epsilon}(a)\in C(\mathbb T^2)_{\epsilon}$
is a continuous section for every $a \in C^*(F_2)$.
\end{theorem}
\begin{proof} We repeat Exel's arguments for reader's convenience. Let $S$ be the set of sections
$$\epsilon\in [0, 2] \mapsto  \phi_{\epsilon}(a)\in C(\mathbb T^2)_{\epsilon}$$
for $a \in C(F_2)$. According to \cite{Dixmier} (Propositions 10.2.3 and 10.3.2) all one needs to check is that
$S$ is a *-subalgebra of the algebra of all sections, that the set of all $s(\epsilon)$ as $s$ runs through
$S$ is dense in $C(\mathbb T^2)_{\epsilon}$ and that $\|s(\epsilon)\|$ is continuous as a function of $\epsilon$ for all $s \in S$.
The first two properties are trivial while the last one follows from Lemma \ref{Lemma1ContinuousField} and Lemma \ref{Lemma2ContinuousField}.
\end{proof}

\bigskip

\subsection{Local minima of commutators}
A problem of characterizing  pairs of unitary operators which are
local minimum points for the commutator norm was discussed in \cite{Exel} from different points of view.
In particular it was proved that each pair of non-commuting unitary operators admits a *-strong dilated perturbation to a pair of unitaries with smaller norm of commutator \cite[Th. 4.2]{Exel}. Precisely, given non-commuting unitaries $U, V\in B(H)$, there are sequences of unitaries $U_n, V_n \in B(H^{(\infty)})$, such that $\|[U_n, V_n]\|< \|[U, V]\|$ and the compressions $proj_H\circ U_n|_H$ and $proj_H\circ V_n|_H$ of $U_n, V_n$ to $H$ converge $\ast$-strongly to $U$ and $ V$ respectively. Below we show that this result follows from Construction   (\ref{eq0}). Moreover the result can be generalized in two directions:

\medskip

1) $H^{(\infty)}$ can be replaced by $H^{(2)}$ and the $\ast$-strong convergence by the operator norm convergence;

\medskip

2) If $H$ is infinite-dimensional, then *-strong dilated perturbations
can be replaced by *-strong perturbations.

\begin{theorem}\label{dilations} Let $U$ and $V$ be non-commuting unitary operators on a Hilbert space $H$. Then there are sequences of unitary operators $U_n, V_n \in B(H\oplus H)$, such that $\|[U_n, V_n]\|< \|[U, V]\|$ and the compressions $proj_H\circ U_n|_H$ and $proj_H\circ V_n|_H$ of $U_n, V_n$ to $H$ converge in norm to $U$ and $ V$ respectively.

If $H$ is infinite-dimensional, then
there are sequences of unitary operators $U_n, V_n$ on $H$ such that $\|[U_n, V_n]\|< \|[U, V]\|$ and  $U_n\to U$ $\ast$-strongly and $V_n\to V$ $\ast$-strongly.
\end{theorem}
\begin{proof}  Using  $U_t, V_t$ with $t<1$ in  Construction \ref{construction}, we see that there are $U_n, V_n$ on $H\oplus H$ such that $\|[U_n, V_n]\|< \|[U, V]\|$ and $U_n \to U\oplus (-U^*)$ and $V_n \to V\oplus V$ in the operator norm topology.

For the second statement we apply a result of Hadwin that states that for any $A, B\in B(H)$,where $H$ is infinite-dimensional, there are unitaries $W_m: H\to H\oplus H$ not depending of $A, B$ such that $W_m^*\left(\begin{array}{cc} A & \\ &B\end{array}\right)W_m\to A$ $\ast$-strongly (this result is described in detail in [\cite{CourtneyShulman}, Lemma 2.4]). Hence for an appropriately chosen $m(n)$, $W_{m(n)}^*U_nW_{m(n)}$ and $W_{m(n)}^*V_nW_{m(n)}$ converge $\ast$-strongly to U and V respectively.
\end{proof}

\begin{remark} In \cite{Exel}  Exel conjectured that a pair  of unitary matrices $u = \oplus u_j$ and $v = \oplus v_j$, with each pair $u_j, v_j$ being irreducible,  is a local minimum of commutators if and only if one has that $u_jv_ju_j^{-1}v_j^{-1}$
is a scalar, for some value of $j$ for which
$\|u_jv_j-v_ju_j \| = \|uv-vu\|$.  A counterexample to the conjecture was constructed by Manuilov  \cite{Manuilov}. Theorem \ref{dilations} gives different counterexamples: any pair of unitary matrices of the form $u= U\oplus (-U^*)$ and $v= V\oplus V$ with $UVU^{-1}V^{-1}$ being a scalar not equal to $1$ is a counterexample.
\end{remark}

\subsection{Very flexible stability}

The  notion of very flexible stability of groups was introduced by Becker and Lyubotzky  \cite{BeckerLyubotzky}. Very flexible stability of C*-algebras is defined similarly and the definitions agree in the sense that $C^*(G)$ is very flexibly stable if and only if $G$ is very flexibly stable.

\medskip

For any Hilbert space $H$, infinite-dimensional or finite-dimensional, with a fixed o.n.b. $\{e_i\}$, the projection onto the span of the first $n$ basis vectors will be denoted by $P_n$. We identify $B(H)$, for $n$-dimensional $H$, with $M_n$.

Let A  be a  C*-algebra. A sequence of maps   $\phi_n: A\to M_{k_n}$, $n\in \mathbb N$, is an {\it approximate representation } with respect to a given norm $\|\;\|$ on matrices, if $$\lim_{n\to \infty}\|\phi_n(ab) - \phi_n(a)\phi_n(b)\| =0,$$   $$\lim_{n\to \infty}\|\phi_n(\lambda a +\mu b) - \lambda \phi_n(a)- \mu \phi_n(b)\| =0,$$ 
$$\lim_{n\to \infty}\|\phi_n(a^*) - \phi_n(a)^*\| =0,$$
for any $a, b\in A, \lambda, \mu \in \mathbb C$.

\begin{definition}  A C*-algebra $A$ is {\it very flexibly HS-stable}  if for any approximate representation $\phi_n: A\to M_{k_n}$, $n\in \mathbb N$,   w.r.t. the normilized HIlbert-Schmidt norm,  there exist $m_n\ge k_n$ and representations $\rho_n: A \to M_{m_n}$, $n\in \mathbb N$, such that
$$  \|\phi_n(a) - P_{k_n}\rho_n(a)P_{k_n}\|_2 \to 0,  \;\text{ as}\; n\to \infty,$$ for any $a\in A$.

A C*-algebra $A$ is {\it very flexibly  operator norm stable}  if for any approximate representation $\phi_n: A\to M_{k_n}$, $n\in \mathbb N$,   w.r.t. the operator  norm,  there exist $m_n\ge k_n$ and representations $\rho_n: A \to M_{m_n}$, $n\in \mathbb N$, such that
$$  \|\phi_n(a) - P_{k_n}\rho_n(a)P_{k_n}\| \to 0,  \;\text{ as}\; n\to \infty,$$ for any $a\in A$.
\end{definition}

It is known that an RFD C*-algebra with the LLP is very flexibly HS-stable (\cite{EckhardtShulman}) and therefore by the results of previous subsections, soft tori are very flexibly HS-stable.   We show  below  that they are also very flexibly operator norm stable.

Let $\omega$  be any non-trivial ultrafilter on $\mathbb N$ and let $\prod_{\omega} (M_{k_n}, \|\|_2)$ be the ultraproduct $C^*$-algebra
$$\prod_{\omega} (M_{k_n}, \|\|_2) = \prod M_{k_n} / \{(T_n)\;|\; \|T_n\|_2 \to_{\omega} 0\}.$$

Very flexible operator norm  stability can be reformulated as the following lifting property.

\begin{definition} A C*-algebra $A$  is very flexibly operator norm  stable if for any $\ast$-homomorphism $f: A \to \prod M_{k_n}/\oplus M_{k_n}$ there is a $\ast$-homomorphism $g: A \to \prod M_{m_n}/\oplus M_{m_n}$ such that $f\oplus g$ lifts to a $\ast$ homomorphism to $\prod M_{k_n+m_n}$.
\end{definition}

The same arguments as in the proof of     Theorem \ref{SoftTorusthe LP}  together with the well known fact that a pointwise limit of liftable $\ast$-homomorphisms to $\prod M_n/\oplus M_n$ is liftable show 

\begin{proposition}\label{SoftTorusVeryFlexiblyStable} Soft torus is very flexibly operator norm stable.
\end{proposition}

\noindent Of course the actual stability, that is not (very) flexible one, is particularly interesting. In \cite{SoftTorusNotSP} it is proved that the soft torus is not operator norm stable (and therefore not semiprojective) when $\epsilon < 2$.
Whether the soft torus is Hilbert-Schmidt stable, for $0<\epsilon<2$,  is not known.

\medskip

We already know from the proof of Theorem \ref{FnxFn} that the softening of $C^*(F_n\times F_n)$
$$ C^*\langle u_i, v_i\;|\;  u_i \; \text{and} \; v_i \; \text{are unitaries}, \;  i= 1, \ldots, n, \; \|[u_i, v_j]\| \le \epsilon,\forall  i, j = 1, \ldots, n \rangle,$$ which will be denoted by $C^*(F_n\times F_n)_{\epsilon}$,    has the LP.
The same arguments as in the proofs of Theorem \ref{SoftTorusRFD}, Theorem \ref{ContinuousField}, Proposition \ref{SoftTorusVeryFlexiblyStable} show that it also shares the other properties of soft tori.

\begin{proposition}\label{OtherPropertiesOfFnxFn} Let $n\in \mathbb N$ and $\epsilon >0$. Then
\begin{enumerate}

\item  $C^*(F_n\times F_n)_{\epsilon}$ is RFD,  very flexibly Hilbert-Schmidt stable, very flexibly operator norm stable, and has the LP,

\medskip

\item  There exists a continuous field of C*-algebras over the interval $[0,2]$ such
that $C^*(F_n\times F_n)_{\epsilon}$ is the fiber over $\epsilon$ and such that
$\epsilon\in [0, 2] \mapsto  \phi_{\epsilon}(a)\in C^*(F_n\times F_n)_{\epsilon}$
is a continuous section for every $a \in C^*(F_{2n})$.

\noindent (Here $\phi_{\epsilon}: C^*(F_{2n}) \to C^*(F_n\times F_n)_{\epsilon}$ is the canonical surjection).
\end{enumerate}
\end{proposition}

\section{ Ext being a group versus the LLP}

It is known that the LLP implies that Ext is a group and it is not known whether they are equivalent (\cite[p.12]{Ozawa}). In this section we prove that for an interesting class of C*-algebras, including $C^*(F_2\times F_2)$ and all contractible C*-algebras, they are equivalent.

Below for an inductive limit $A =\varinjlim P_n$  we denote the connecting maps by  $\theta_{n, m}: P_n \to P_m$ and $\theta_{n, \infty}: P_n \to A$.

The following two lemmas are similar to \cite[Lemmas 13.4.4, 13.4.5]{BO}.

\begin{lemma}\label{ExtLemma1} Let $A = \varinjlim P_n$,  where all $P_n$ have the LP and all the connecting maps are surjective.   Let $J = \ker \theta_{0, \infty}$.  Then for any C*-algebra $D$, in the  diagram

$$ \begin{tikzcd} P_0 \otimes D \arrow{r}{\psi\otimes id} \arrow{d} & \left(\prod P_n\right)\otimes D \arrow{d}  \\ \left(P_0\otimes D\right) / \left(J \otimes D\right)  \arrow[r, dashed, "\theta"] & \frac{ (\prod P_n)\otimes D}{ (\oplus P_n)\otimes D}
 \end{tikzcd}$$

\noindent the induced $\ast$-homomorphism $\theta$  is isometric. Here $\psi = \prod_n \theta_{0, n}$ and   $\otimes = \otimes_{min}$.
\end{lemma}

\begin{proof} Since $\psi(J) \subset \oplus P_n$, $\theta$ is well-defined.
    Since each $P_n$ has the LP, there is a ccp section $s_n: P_n\to P_0$ of the connecting homomorphism $\theta_{0, n}$.
    Let $\sum_{i=1}^N b_i\otimes d_i + J\otimes D \in \left(P_0\otimes D\right) / \left(J\otimes D\right)$. Since $\ker \theta_{0, n}\subseteq J$ and
    $s(\theta_{0, n}(b_i))-b_i\in \ker \theta_{0, n}$, we have
    \begin{multline*} \|\sum_{i=1}^N b_i\otimes d_i + J \otimes D\| \le \|\sum_{i=1}^N b_i\otimes d_i + \ker \theta_{0, n} \otimes D\|
   \\ \le  \|\sum_{i=1}^N s_n(\theta_{0, n}(b_i))\otimes d_i + \ker \theta_{0, n}\otimes D\|
   \\ \le \|\sum_{i=1}^N s_n(\theta_{0, n}(b_i))\otimes d_i \| = \|(s_n\otimes id_D)(\sum_{i=1}^N \theta_{0, n}(b_i)\otimes d_i)\|
    \le \|\sum_{i=1}^N \theta_{0, n}(b_i)\otimes d_i\| .\end{multline*}
    (where to obtain  the last equality we used that $s\otimes id_D$ is well-defined and contractive \cite[Th. 3.5.3]{BO}). It follows that
    $$ \|\sum_{i=1}^N b_i\otimes d_i + J \otimes D\| \le \limsup_n \|\sum_{i=1}^N \theta_{0, n}(b_i)\otimes d_i\|.$$

    On the other hand
    \begin{multline*}\theta\left(\sum_{i=1}^N b_i\otimes d_i + J\otimes D\right)\| = \|\sum_{i=1}^N \left(\theta_{0, n}(b_i)\right)_{n\in \mathbb N} \otimes d_i +
    (\oplus P_n)\otimes D\| \\ = \limsup_n \|\sum_{i=1}^N \theta_{0, n}(b_i)\otimes d_i\| \end{multline*}
    (here we used the well-known and not difficult fact that there are isometric embeddings $(\prod A_n)\otimes D \hookrightarrow \prod (A_n\otimes D)$ and
    $(\oplus A_n)\otimes D \hookrightarrow \oplus (A_n\otimes D)$).

    Therefore $\theta$ is an isometry.
\end{proof}

\begin{lemma}\label{ExtLemma2} Let $A = \varinjlim P_n$, where all $P_n$ are separable RFD with the LP and all the connecting maps are surjective.  Let $\theta$ and $\psi$ be as in Lemma \ref{ExtLemma1}. Then  for
    any separable C*-algebra $D$,  there exist finite-dimensional representations $\sigma_n: P_n \to M_{k_n}$ such that    in the following diagram
    $$ \begin{tikzcd} P_0 \otimes D \arrow{r}{\psi\otimes id} \arrow{d} & \left(\prod P_n\right) \otimes D \arrow{d}  \arrow{r}{\left(\prod \sigma_n\right)\otimes id} & \;\;\; \left(\prod M_{k_n}\right) \otimes D
\arrow{d}\\ \left(P_0\otimes D\right) / \left(J \otimes D\right)  \arrow{r}{\theta}  \arrow[bend right=30, dashed, swap]{rr}{\pi}& \frac{ (\prod P_n)\otimes D}{ (\oplus P_n)\otimes D}  \arrow[r, dashed] & \frac{\left(\prod M_{k_n}\right) \otimes D}{\left(\oplus M_{k_n}\right) \otimes D}
 \end{tikzcd}$$
    the induced $\ast$-homomorphism $\pi$ is isometric.
\end{lemma}
\begin{proof} Let $\{\sum_{l=1}^{N_i} b_l^{(i)} \otimes d_l^{(i)}, i\in \mathbb N\}$ be a dense subset in $P_0\otimes D$.

    \medskip

    {\it Claim:}  For $n\in \mathbb N$, there exists a finite-dimensional representation $\sigma_n$ of $P_n$ such that
    $$\| \sum_{l=1}^{N_i} \theta_{0, n}(b_l^{(i)}) \otimes d_l^{(i)} \| \le
    \| \sum_{l=1}^{N_i} \sigma_n(\theta_{0, n}(b_l^{(i)}))\otimes d_l^{(i)} \| + \frac{1}{n}, $$ for $i= 1, \ldots, n$.

\medskip

    {\it Proof of Claim:} By definition of the minimal tensor product, there are a representation $\sigma$ of $P_n$ and a representation
    $\rho$ of $D$ such that
    \begin{equation}\label{Lemma2Equation1}\| \sum_{l=1}^{N_i} \theta_{0, n}(b_l^{(i)}) \otimes d_l^{(i)} \| \le \| \sum_{l=1}^{N_i} \sigma(\theta_{0, n}(b_l^{(i)})) \otimes
    \rho(d_l^{(i)}) \|  + \frac{1}{2n}, \end{equation} for $i=1, \ldots, n$.
    Since $P_n$ is RFD, by Exel-Loring theorem (\cite{ExelLoring}) there exist representations $\sigma_k$ of $P_n$ living on finite-dimensional subspaces of
    the space of $\sigma$ such that $\sigma_k\to \sigma$ in the $\ast$-strong operator topology pointwisely.   Then $\sigma_k\otimes \rho \to \sigma\otimes \rho$  in the $\ast$-strong operator topology
     pointwisely. Hence there is $k_n$ such that for $\sigma_{k_n}$, which we will denote by $\sigma_n$ for short, we have

     \begin{equation}\label{Lemma2Equation2}\| \sum_{l=1}^{N_i} \sigma(\theta_{0, n}(b_l^{(i)})) \otimes
     \rho(d_l^{(i)}) \|   \le \| \sum_{l=1}^{N_i} \sigma_n(\theta_{0, n}(b_l^{(i)})) \otimes
     \rho(d_l^{(i)}) \|  + \frac{1}{2n},\end{equation} for $i= 1, \ldots, n$.
     From (\ref{Lemma2Equation1}) and (\ref{Lemma2Equation2}) we obtain
     \begin{multline*}\| \sum_{l=1}^{N_i} \theta_{0, n}(b_l^{(i)}) \otimes d_l^{(i)} \| \le \| \sum_{l=1}^{N_i} \sigma_n(\theta_{0, n}(b_l^{(i)}))
      \otimes \rho(d_l^{(i)}) \| + \frac{1}{n} \\ =
      \|(id_{\sigma_n(P_n)}\otimes \rho)\left( \sum_{l=1}^{N_i} \sigma_n(\theta_{0, n}(b_l^{(i)}))\otimes d_l^{(i)} \right)\| +\frac{1}{n}
      \le \| \sum_{l=1}^{N_i} \sigma_n(\theta_{0, n}(b_l^{(i)}) \otimes d_l^{(i)}) \|  + \frac{1}{n}, \end{multline*}
      when $i = 1, \ldots, n$. Claim is proved.

      \medskip
      \noindent       Then
      \begin{multline*}\|\pi\left(\sum_{l=1}^{N_i} b_l^{(i)} \otimes d_l^{(i)} + J\otimes D\right)\| = \|  \sum_{l=1}^{N_i} \left(\sigma_n\left(\theta_{0, n}(b_l^{(i)})\right)\right)_{n\in \mathbb N} \otimes d_l^{(i)}
      + \left(\oplus M_{k_n}\right) \otimes D\| \\ = \limsup_n \| \sum_{l=1}^{N_i} (\sigma_n\theta_{0, n})(b_l^{(i)}) \otimes d_l^{(i)} \|  \stackrel{Claim}{=}
      \limsup_n  \| \sum_{l=1}^{N_i} \theta_{0, n}(b_l^{(i)}) \otimes d_l^{(i)} \| \\ =
      \|\sum_{l=1}^{N_i} \left(\theta_{0, n}(b_l^{(i)})\right)_{n\in \mathbb N} \otimes d_l^{(i)} + \left(\oplus P_n\right)\otimes D\|
= \|\theta\left(\sum_{l=1}^{N_i} b_l^{(i)} \otimes d_l^{(i)} + J\otimes D\right) \| \\ \stackrel{Lemma \; \ref{ExtLemma1}}{=} \|\sum_{l=1}^{N_i} b_l^{(i)} \otimes d_l^{(i)} + J\otimes D\|.\end{multline*}
\end{proof}

\begin{theorem}\label{LLPvsExt} Suppose $A$ is an inductive limit, with surjective connecting maps, of separable C*-algebras that are RFD and have the LP. Then $A$ has the LLP if and only if $Ext(A)$ is a group.
\end{theorem}
\begin{proof} The  "only if"  is known, so we assume that $Ext(A)$ is a group.
We write $A = \varinjlim P_n$, where all $P_n$ are separable RFD with the LP and all the connecting maps are surjective. Let $D$ be a separable  C*-algebra. 
Let $k_n$, $n\in \mathbb N$,  be the numbers obtained by applying Lemma \ref{ExtLemma2} to the C*-algebra $D\oplus \mathbb C$.   
 Since $$\left(P_0\otimes (D\oplus \mathbb C)\right)/\left(J\otimes (D\oplus \mathbb C)\right) \cong \left(\left(P_0\otimes D\right)/\left(J\otimes D\right) \right)\oplus \left(P_0/J\right),$$
 Lemma \ref{ExtLemma2} gives us two embeddings: 
$$ \gamma :=A \cong P_0/J \to  \prod M_{k_n}/\oplus M_{k_n}$$ and 
$$\pi:= (P_0\otimes D) / (J\otimes D)  \to \left( \prod M_{k_n}\otimes D\right) /\left(\oplus M_{k_n} \otimes D\right).$$  
Embed $\prod M_{k_n} $ into $B(H)$ diagonally, then $\oplus M_{k_n} \subset K(H)$  and  $\prod M_{k_n}/\oplus M_{k_n} \subset B(H)/K(H)$. The latter embedding will be denoted by $j$.

 Then $j\circ \gamma: A \to B(H)/K(H)$ is an extension. Since $Ext(A)$ is a group, $j\circ \gamma$
 lifts to a completely positive map $\tau: A \to B(H)$. Let $E: B(H) \to \prod M_{k_n}$ be the canonical expectation (that is the one that sends an operator to its block-diagonal). Then $E\circ \tau$ is a ccp lift of $\gamma$.

\noindent Let $q: \prod M_{k_n} \to \prod M_{k_n}/\oplus M_{k_n}$ be the canonical surjection and  let \begin{equation}\label{Ext1} B= q^{-1}(\gamma(A)).\end{equation}   Then $E\circ \tau$   must have range in $B$. Thus $\gamma: A\to B/\oplus M_{k_n} $ is liftable, hence locally liftable. By Effros-Haagerup Theorem  \cite[Th. 13.1.6]{BO},  we have
\begin{equation}\label{ExtCanonically} A\otimes D =   (B/\oplus M_{k_n})\otimes D \cong(B\otimes D) /(\oplus M_{k_n} \otimes D)\end{equation} canonically.
 By (\ref{Ext1}) $\pi$ has range in $\left( B\otimes D\right) /\left(\oplus M_{k_n} \otimes D\right)$. Let $A\odot D$ denote the algebraic tensor product of $A$ and $D$.  We obtain the embedding
\begin{equation}\label{Ext4} A\odot D \subseteq (P_0\otimes D) / (J\otimes D) \stackrel{\pi} {\to}  \left( B\otimes D\right) /\left(\oplus M_{k_n} \otimes D\right) \stackrel{(\ref{ExtCanonically})}{\cong} A\otimes D \end{equation}
such that
\begin{equation}\label{Ext3}a\odot d \mapsto \gamma(a)\otimes d,\end{equation} for $a\in A, d\in D$. Let $\|\;\|_{ind}$ be the C*-norm induced by this embedding on $A\odot D$. Then  by (\ref{Ext3})
\begin{multline*}\|\sum_{i=1}^n a_i\odot d_i\|_{ind} = \|\sum_{i=1}^n \gamma(a_i)\otimes d_i\|_{min} = \|\left(\gamma\otimes id\right)\left(\sum_{i=1}^na_i\otimes d_i\right)\|_{min} \\ \le  \|\sum_{i=1}^na_i\otimes d_i\|_{min}. \end{multline*}
By minimality of the minimal tensor product, the induced norm on $A\odot D$ is the minimal one. Therefore, by (\ref{Ext4}),
$$ (P_0\otimes D) / (J\otimes D) \cong A\otimes D.$$  Since $D$ was an arbitrary separable C*-algebra, again by Effros-Haagerup Theorem\footnote{We note that in Effros-Haagerup Theorem  \cite[Th. 13.1.6]{BO} in  the condition (1) any C*-algebra $A$ can be replaced by any separable C*-algebra $A$ as proved in \cite[Th. 4.3, $(i)\Longleftrightarrow(vi)$]{Huruya}. } we conclude that the identity map on $A = P_0/J$ is locally liftable. Together with the fact that  $P_0$ has the LLP it implies easily that $A$ has the LLP. Indeed, let $\phi: A \to C/I$ be a cp map and let $E\subset A$ be a finite-dimensional operator system.  Let $\psi_0: P_0\to A$ be the canonical surjection. Let $\delta_1: E\to P_0$ be a cp lift of     the identity map on $A = P_0/J$.     Since $P_0$ has the LLP,  $(\phi\circ \psi_0)|_{\delta_1(E)}$ lifts to a cp map $\delta_2: \delta_1(E) \to C$.  Then $\delta_2 \circ \delta_1$ is a cp lift of $\phi|_E$.
\end{proof}

We will see now that the assumptions of Theorem \ref{LLPvsExt} hold for a large class of C*-algebras.

\begin{lemma}\label{LPpassToIdeals} Let $A$ be a separable C*-algebra that has the (L)LP and let $J$ be  an ideal of $A$.
Then $J$ has the (L)LP.
\end{lemma}
\begin{proof} We will show that if every cp map from $A$ to a quotient C*-algebra $B/I$ lifts, then every cp map from $J$ to $B/I$ lifts. It will imply that the LP passes to ideals and, since the LLP is equivalent to lifting cp maps to the Calking algebra, it also will imply that the LLP passes to ideals.

Let $\phi: J \to B/I$ be a cp map. Let $\{i_{\lambda}\}$ be  a quasicentral approximate unit for $J\lhd A$. We define a ccp map $\rho_{\lambda}: A\to J$ by $$\rho_{\lambda}(a) = i_{\lambda}^{1/2}ai_{\lambda}^{1/2}, $$ for $a\in A$. Then
$$\phi\circ\rho_{\lambda}(i) = \phi(i_{\lambda}^{1/2}ii_{\lambda}^{1/2}) \to \phi(i),$$ for any $i\in J$. Since $\phi\circ\rho_{\lambda}$ is liftable by assumption, $\phi$ is a pointwise limit of liftable maps $\phi\circ\rho_{\lambda}|_J$ and therefore is liftable itself by Arveson's theorem.
\end{proof}

\begin{proposition}\label{suspension} For any separable C*-algebra $A$, its suspension $SA$ is an inductive limit, with surjective connecting maps, of separable C*-algebras that are RFD and have the LP.
\end{proposition}
\begin{proof} By \cite{LoringShulman} we can write $CA$ as an inductive limit $CA = \varinjlim B_n$, where all $B_n$ are projective, hence are RFD  and have the LP, and  all the connecting maps are surjective, which in particular implies that all the canonical maps  $\theta_{n, \infty}: B_n \to CA$ are surjective. Since $SA$ is a C*-subalgebra (in fact, an ideal) of $CA$,
$$SA = \varinjlim \theta_{n, \infty}^{-1}(SA)$$
(\cite[Lemma 13.1.5]{LoringBook}), with the surjective connecting maps $\theta_{n, \infty}|_{\theta_{n, \infty}^{-1}(SA)}$.  Since $ \theta_{n, \infty}^{-1}(SA)$ is an ideal in $B_n$ and, by Lemma \ref{LPpassToIdeals}, the LP property passes to ideals, $ \theta_{n, \infty}^{-1}(SA)$ has the LP. Since the RFD property obviously passes to C*-subalgebras,
$ \theta_{n, \infty}^{-1}(SA)$ is RFD.
\end{proof}

\begin{corollary}\label{examplesLLPExt} Suppose $A = C^*(F_2\times F_2)$ or $A$ is any separable contractible C*-algebra or $A$ is a suspension of any separable C*-algebra. Then $A$ has the LLP if and only if $Ext(A)$ is a group.
\end{corollary}
\begin{proof}
$A = C^*(F_2\times F_2)$ satisfies the assumptions of Theorem \ref{LLPvsExt} by Theorem \ref{FnxFn} and Proposition \ref{OtherPropertiesOfFnxFn}.  Any separable contractible C*-algebra  satisfies the assumptions of Theorem \ref{LLPvsExt} because it is an inductive limit of projective C*-algebras \cite{Thiel} and projectivity implies the LP and RFD. The suspension over any separable C*-algebra satisfies the assumptions of Theorem \ref{LLPvsExt} by Proposition \ref{suspension}.
\end{proof}

\begin{corollary} (Kirchberg \cite{Kirchberg} in unital case) For any separable C*-algebra $A$, $A$ has the LLP if and only if $Ext(SA)$  is a group if and only if $Ext(CA)$ is a group.
\end{corollary}
\begin{proof} We only need to show that $A$ has  the LLP if and only if $SA$ has the LLP if and only if $CA$ has the LLP. The statement will then follow from Corollary \ref{examplesLLPExt}.

Since the LLP passes to tensor products with nuclear C*-algebras (\cite{Kirchberg}), if $A$ has the LLP, so do $SA$ and $CA$.

To prove the other implication, let $f\in C_0(0, 1)$ ($f\in C_0(0, 1]$, respectively) be  a positive function such that $f(1/2)=1$. We define a cp map $\delta: A \to SA$ ($\delta: A \to CA$, respectively) by $$\delta(a) = a\otimes f,$$ for any $a\in A$. Then $$id_A = ev_{1/2}\circ \delta.$$ Therefore  every cp map $\phi: A \to Q(H)$ factorizes through $SA$ ($CA$, respectively). It implies that if $SA$ ($CA$, respectively) has the LLP, so does $A$.
\end{proof}

\medskip

\section{Miscellaneous}

\subsection{Extensions with the WEP}

We embed $\prod M_n$ into $B(H)$ and $\oplus M_n$ into $K(H)$ diagonally.

\begin{lemma}\label{WEPlemma} Let $D$ be any C*-algebra. Then $$(K\otimes D) \bigcap \left(\left(\prod M_n\right)\otimes D\right) = \left(\oplus M_n\right)\otimes D.$$
\end{lemma}
\begin{proof} Let $a\in (K\otimes D) \bigcap ((\prod M_n)\otimes D)$. Given $\epsilon > 0$, there is $
\sum_{i=1}^m k_i\otimes f_i \in K\otimes D$ and $\sum_{j=1}^n T_j\otimes \hat f_j \in (\prod M_n)\otimes D$ such that
\begin{equation}\label{WEP1} \|a- \sum_{i=1}^m k_i\otimes f_i \|\le \epsilon, \; \|a- \sum_{j=1}^n T_j\otimes \hat f_j\|\le \epsilon.\end{equation} Hence
\begin{equation}\label{WEP2} \|\sum_{i=1}^m k_i\otimes f_i  - \sum_{j=1}^n T_j\otimes \hat f_j\|\le 2\epsilon.\end{equation}
Let $E: B(H) \to \prod M_n$ be the map that sends an operator to its block-diagonal.  Then $E\otimes id: B(H)\otimes D \to (\prod M_n)\otimes D$ is contractive by \cite[Th. 3.5.3]{BO}. Applying $E\otimes id$ to (\ref{WEP2}) we obtain
\begin{equation}\label{WEP3} \| \sum_{i=1}^m diag(k_i) \otimes f_i - \sum_{j=1}^n T_j\otimes \hat f_j\| \le 2\epsilon.\end{equation}
We conclude from (\ref{WEP1}) and (\ref{WEP3}) that
$$\|\sum_{i=1}^m diag(k_i)\otimes f_i - a\|\le 3\epsilon.$$
 Since $diag(k_i) \in \oplus M_n$,  $a$ is approximated by elements of $(\oplus M_n)\otimes D$.
\end{proof}

The following theorem generalizes Kirchberg's result about cones \cite[Th. 13.4.1]{BO}.

\begin{theorem} Let $A$ be  separable contractible (or, more generally, an inductive limit, with surjective connecting maps,  of separable RFD  C*-algebras with the LP) with the QWEP. Then there exists a quasidiagonal
    extension
    $$0\to K \to B \to A \to 0$$  such that $B$ has the WEP.
\end{theorem}
\begin{proof}
    Let $F$ be the full group C*-algebra of the free group with countably many generators. Replacing $A$ by its unitization if necessary,       we can write $A = \varinjlim P_n$ with $P_n$ being separable  RFD with the LP and $P_0 = F$.
We write $A = F/J$ and let $\psi$ and $\sigma_n, n\in \mathbb N, $ be as in Lemma \ref{ExtLemma2} applied to $D=F$. Let $$\Psi = \left(\prod \sigma_n\right) \circ \psi: F \to \prod M_n.$$
Let $$B = C^*(\Psi(F), K).$$ Then $B/K = A$.

\medskip

{\it Claim:} $\left(B\otimes F\right) /\left( K \otimes F\right)$ is canonically isomorphic to $A\otimes_{max} F$.

\medskip

{\it Proof of Claim:}  By \cite[Prop. 3.6.1]{BO} $(\oplus M_n)\otimes F \subset K\otimes F$, $(\oplus M_n)\otimes F \subset B(H)\otimes F$. Together with  Lemma \ref{WEPlemma} this implies $$\left(\left(\prod M_n\right)\otimes F \right)/ \left((\oplus M_n)\otimes F\right) \subset \left(B(H)\otimes F\right)/ \left(K\otimes F\right).$$

Since $A$ is QWEP, by \cite[exercise 13.3.5]{BO} (applied to $ C_1=F, C_2=F, J_1=J, J_2=0$)
$$A\otimes_{max}F = \left( F\otimes F\right) /\left(J\otimes F\right) $$ canonically.  Then using Lemma \ref{ExtLemma2}  we obtain
$$A\otimes_{max} F = \left(F\otimes F\right)  /\left(J\otimes F \right)  \cong \theta(\left(F\otimes F\right)  /\left(J\otimes F\right) ) \subseteq \left(B\otimes F\right)  / \left( K \otimes F\right) .$$ By the construction of all the maps here we have
$$\sum_{i=1}^N f_i\otimes \tilde f_i + J\otimes F \mapsto \sum_{i=1}^N \Psi(f_i)\otimes \tilde f_i + K\otimes F \in \left(B\otimes F\right)/\left(K\otimes F\right)$$ and since $B$ is generated by $\Psi(F)$ and $K$, the range is the whole $\left(B\otimes F\right)/\left(K\otimes F\right)$.
Claim is proved.

\noindent Now we have a commuting diagram

  $$\begin{tikzcd}
     0\arrow{r} & K\otimes_{max} F \arrow{r}\arrow{d} & B\otimes_{max} F \arrow{r} \arrow{d}& A\otimes_{max} F \arrow{r}  \arrow{d} & 0      \\  0\arrow{r}  &K\otimes F \arrow{r} &B\otimes F \arrow{r} &A\otimes_{max} F \arrow{r} & 0
  \end{tikzcd}$$

\medskip

\noindent with both rows being exact. Since $K$ is nuclear, by 5 Lemma  we obtain $B\otimes_{max} F = B\otimes F$.
By \cite[Cor. 13.2.5]{BO} this is equivalent to $B$ having the WEP.
\end{proof}

\medskip
\subsection{Group C*-algebras}

Here we give a new proof to the fact that the class of C*-algebras with the LP is closed under crossed products with amenable groups which was proved by Buss, Echterhoff and Willett \cite[Th. 7.4]{BEW}.
%The following theorem gives some new examples of C*-algebras with the LP, a simple result that seems to be overlooked.

\begin{theorem} (Buss, Echterhoff, Willett) Suppose $G$ is a discrete countable amenable group  and $A$ is a separable C*-algebra with the LP. Then $A\rtimes_r G$ has the LP.
\end{theorem}
\begin{proof} By the proof of \cite[Th. 4.2.6(2)]{BO}  there exist cpc maps $\phi_n: A\rtimes_r G \to M_{k_n}\otimes A$  and $\psi_n:  M_{k_n} \otimes A \to A\rtimes_r G$   such that $id_{A\rtimes_r G}$ is the pointwise limit of $\psi_n\circ\phi_n$.

Let $\delta: A\rtimes_r G \to B/I$ be a ccp map. Since $M_n\otimes A$ has the LP, $\delta\circ \psi_n$ is liftable, and then $\delta\circ\psi_n\circ \phi_n$ is liftable, for any $n$. Then $\delta$ is a pointwise limit of liftable ccp maps and by Arveson's theorem it is itself liftable.
\end{proof}

\begin{corollary}(Buss, Echterhoff, Willett) Suppose $G$ is a discrete countable amenable group  and $A$ is a separable C*-algebra with the LP. Then $A\rtimes G$ has the LP.
\end{corollary}
\begin{proof} Since $G$ is amenable, by \cite[Th. 4.2.6(1)]{BO} $A\rtimes_r G = A\rtimes G$, and the statement follows from the theorem above.
\end{proof}

In particular the following are interesting examples of full group C*-algebras with the LP.

\begin{corollary} $C^*(F_n\rtimes G)$ has the LP, for any discrete countable amenable group $G$. \end{corollary}

\bigskip

\textbf{ Conflict of interest statement.}
 On behalf of all authors, the corresponding author states that there is no conflict of interest.

\end{document}